\documentclass[a4paper]{amsart}
\usepackage[all]{xy}
\usepackage[colorlinks=true]{hyperref}
\usepackage{enumerate}
\usepackage{amssymb}

\newtheorem{thm}{Theorem}[section]
\newtheorem{prop}[thm]{Proposition}
\newtheorem{lem}[thm]{Lemma}
\newtheorem{cor}[thm]{Corollary}

\theoremstyle{definition}
\newtheorem{defn}[thm]{Definition}
\newtheorem{ex}[thm]{Example}
\theoremstyle{remark}
\newtheorem{rem}[thm]{Remark}

\newcommand{\Rr}{\mathbb R}
\newcommand{\Ss}{\mathbb S}

\newcommand{\Cc}{\mathbb C}

\renewcommand{\d}{\mathrm d}

\newcommand{\set}[1]{\left\{#1\right\}}
\newcommand{\eval}[1]{\left\langle#1\right\rangle}
\newcommand{\brr}[1]{\left[#1\right]}

\newcommand{\rmap}{\longrightarrow}

\newcommand{\X}{\ensuremath{\mathfrak{X}}}
\newcommand{\F}{\ensuremath{\mathcal{F}}}

\newcommand{\ds}{\displaystyle}

\renewcommand{\d}{\mathrm d}               % differential
\newcommand{\Lie}{\boldsymbol{\pounds}}            % Lie derivative
\newcommand{\smalcirc}{\mbox{\,\tiny{$\circ $}\,}}  %small composition circle

\DeclareMathOperator{\tr}{tr}           % trace
\DeclareMathOperator{\red}{red}         % reduced
\DeclareMathOperator{\graf}{graph}	    % graph
\DeclareMathOperator{\ch}{char}         % characteristic class
\DeclareMathOperator{\modular}{mod}     % modular class
\DeclareMathOperator{\VE}{VE}           % van Est map
\renewcommand{\mod}{\modular}
\renewcommand{\top}{\text{\rm top}}
\newcommand{\diff}{\text{\rm diff}}
\newcommand{\proj}{\text{\rm proj}}
\newcommand{\al}{\alpha}
\newcommand{\be}{\beta}

\newcommand{\G}{\mathcal{G}}            % Lie groupoid
\newcommand{\s}{\mathbf{s}}             % source map
\renewcommand{\t}{\mathbf{t}}           % target map
\renewcommand{\H}{\mathcal{H}}          % Lie subgroupoid
\renewcommand{\gg}{\mathfrak{g}}        % Lie algebra
\newcommand{\Lcal}{\mathcal{L}}

\newcommand{\tto}{\rightrightarrows}    % Arrows of a groupoid

\DeclareMathOperator{\ad}{ad}           % adjoint
\DeclareMathOperator{\Ad}{Ad}           % Adjoint
\DeclareMathOperator{\Rep}{Rep}         % Semi-ring of representation
\DeclareMathOperator{\Ver}{Vert}        % Vertical subbundle
\usepackage[dvips]{graphicx}

\begin{document}

\title{The modular class of a Poisson map}

\author{Raquel Caseiro}
\address{Departamento de Matem\'{a}tica\\
Faculdade de Ci\^{e}ncias e Tecnologia\\
Universidade de Coimbra\\
Apartado 3008\\
3001-454 Coimbra\\ Portugal}
\email{raquel@mat.uc.pt}
\author{Rui Loja Fernandes}
\address{Departamento de Matem\'{a}tica\\
Instituto Superior T\'{e}cnico\\1049-001 Lisboa\\ Portugal}
\email{rfern@math.ist.utl.pt}

\thanks{RC is partially supported by PTDC/MAT/099880/2008. RLF is partially supported by the FCT through the Program POCI 2010/FEDER and by project PTDC/MAT/098936/2008. Both authors are also supported by Portugal-France Pessoa agreement 441.00 FRAN\c{C}A-2010/2011.}

%\date{October, 2010}

\begin{abstract}
We introduce the modular class of a Poisson map.
We look at several examples and we use the modular classes of
Poisson maps to study the behavior of the modular class of a Poisson manifold
under different kinds of reduction. We also discuss their
 symplectic groupoid version, which lives in groupoid cohomology.
\end{abstract}

%%% ----------------------------------------------------------------------
\maketitle
%%% ----------------------------------------------------------------------

%%%%%%%%%%%%%%%%%%%%%%%%%%%%%%%%%%%%
%%%%%%%%%%%%%%%%%%%%%%%%%%%%%%%%%%%%
%%%%%%%%%%%%%%%%%%%%%%%%%%%%%%%%%%%%
\section*{Introduction}             %
\label{sec:introduction}           %
%%%%%%%%%%%%%%%%%%%%%%%%%%%%%%%%%%%%
%%%%%%%%%%%%%%%%%%%%%%%%%%%%%%%%%%%%
%%%%%%%%%%%%%%%%%%%%%%%%%%%%%%%%%%%%

The modular class is arguably one of the most basic global invariants in Poisson geometry. For a Poisson manifold $M$,
its modular class $\mod(M)$ is an element of the first Poisson cohomology group $H^1_\pi(M)$, which measures
the obstruction to the existence of a measure in $M$ invariant under all hamiltonian diffeomorphisms (\cite{Koszul,Weinstein1}).
In this work we study how the modular class behaves under a Poisson morphism. This study is motivated by such basic
questions as what happens to the modular class under a symmetry reduction or under restriction to a submanifold.

The notion of modular class extends to Lie algebroids (see \cite{ELW}). Given a Poisson manifold $M$, its cotangent bundle $T^*M$ has a natural Lie algebroid structure and $\mod(M)$ agrees, up to a factor, with the modular class of $T^*M$. The modular class of a Lie algebroid morphism was introduced and studied in \cite{GMM,KW,KLW}). Note, however, that a Poisson map $\phi:M\to N$, in general, \emph{does not}
induce a morphism between the cotangent Lie algebroids $T^*M$ and $T^*N$, nor even a map at the level of Poisson cohomology. This
lack of functorial behavior of a Poisson map makes the definition of its modular class less obvious.

In order to define the modular class of a Poisson map we will use coisotropic calculus and this will bring us back into the realm of Lie algebroids. Recall that a map $\phi:M\to N$ between two Poisson manifolds is Poisson if and only if its graph is a
coisotropic submanifold of $M\times \overline{N}$ (\cite{Weinstein2}). This condition can also be expressed by saying that the conormal bundle
to the graph is a Lie subalgebroid of the product Lie algebroid  $T^*M\times T^*\overline{N}$. Equivalently, the pullback bundle $\phi^*T^*N$ carries a natural Lie algebroid structure and we denote its Lie algebroid cohomology by $H^\bullet_\pi(\phi)$. We will define the \emph{modular class} of $\phi$ to be a certain Lie algebroid cohomology class $\mod(\phi)\in H^1_\pi(\phi)$. This class $\mod(\phi)$ will have representatives certain vector fields along the Poisson map $\phi$ and we will see that it is the obstruction to the existence of measures $\mu$ in $M$ and $\nu$ in $N$, such that the associated modular vector fields are $\phi$-related.

This paper is organized as follows. After introducing and establishing basic properties of the modular class of a Poisson map in Section \ref{sec:Poisson:map}, we look at various applications of this notion. First, in Section \ref{sec:Poisson:submanfld}, we consider Poisson immersions and Poisson submanifolds. We introduce the \emph{relative modular class} $\mod_N(M)$ of a Poisson submanifold $N\subset M$ to be the  opposite  of the modular class of the inclusion. Then we establish the following relationship between the relative modular class  of $N$  and the linear Poisson holonomy $h_N(a):\nu(N)_x\to\nu(N)_y$ along a cotangent loop $a:I\to T^*_NM$:

\begin{thm}
For any cotangent loop $a:I\to T^*_N M$ on a Poisson submanifold $N\subset M$ one has:
\[
\det h_N(a)=\exp\left(\int_a \mod_N(M)\right).
\]
\end{thm}

This extends a result of Ginzburg and Golubev \cite{GG}, who considered the special case of a symplectic leaf.

In Section \ref{sec:reduction}, we turn to Poisson submersions and reduction. First, we show how one can compute the modular class of a Poisson submersion as a characteristic class of a certain representation and we apply this result to conclude that for a free and proper Poisson action $G\times M\to M$ the modular class of the quotient map $M\to M/G$ vanishes. This vanishing result allows to
understand the behavior of the modular class under reduction. In this
direction we prove two results. First, for any free and proper Poisson
action, we consider the Poisson cohomology of projectable multivector fields, denoted by $H^\bullet_{\proj,\pi}(M)$, for which there are maps into the ordinary Poisson cohomology of $M$ and $M/G$:
\[
\xymatrix{
    & H^\bullet_{\pi}(M)\\
H^\bullet_{\proj,\pi}(M)\ar[ur]\ar[dr]& \\
    & H^\bullet_{\pi}(M/G)}
\]
and we show that:

\begin{thm}
For a proper and free Poisson action there is a cohomology class $\mod(M,G)\in H^1_{\proj,\pi}(M)$, which in the diagram above maps to $\mod(M)\in H^1_\pi(M)$ and to $\mod(M/G)\in H^1_\pi(M/G)$.
\end{thm}

We call $\mod(M,G)$ the \textbf{modular class of the Poisson action}. This approach allows us to understand the modular classes of many examples of Poisson quotients. If $\mod(M,G)=0$ the theorem guarantees that $\mod(M/G)=0$. For example, we show that for an action of a compact group by Poisson diffeomorphisms, $\mod(M)=0$ implies that $\mod(M,G)=0$ and hence that $\mod(M/G)=0$. As a consequence, only Poisson manifolds with zero modular class can be quotients of symplectic manifolds by compact group actions. On the other hand, we give examples where $G$ is unimodular, $\mod(M)=0$ but $\mod(M/G)\not=0$ (and, of course, the modular class of the action is non-zero).

We also study Hamiltonian actions admitting a moment map $k:M\to G^*$. In this case, we compute the class $\mod(k)$ in terms of the
adjoint character of the Lie algebra $\gg$ and this allows us to show that if $G$ is unimodular and $\mod(M)=0$ then $\mod(M/G)=0$. In order to understand the modular class of the hamiltonian quotient $M//G:=k^{-1}(e)/G$ we introduce the Poisson cohomology $H^\bullet_{\pi,\proj}(k)$ of multivector fields which are projectable and tangent to the fiber $k^{-1}(e)$, for which there are maps into the ordinary Poisson cohomology of $M$ and $M//G$:
\[
\xymatrix{
    & H^\bullet_{\pi}(M)\\
H^\bullet_{\pi,\proj}(k)\ar[ur]\ar[dr]& \\
    & H^\bullet_{\pi}(M//G)}
\]
We show that:

\begin{thm}
Let $G\times M\to M$ be a proper and free hamiltonian action. There is a class $\mod_G(M)\in H^1_{\pi,\proj}(k)$ which in the diagram above maps to the modular classes $\mod(M)\in H^1_{\pi}(M)$ and $\mod(M//G)\in H^1_{\pi}(M//G)$.
\end{thm}

We call the class  $\mod_G(M)$ the \textbf{equivariant modular class} of the hamiltonian action. Again, this approach allows us to compute the modular class of the hamiltonian quotient $M//G$ in many examples. For example, we show that for an action of a compact group by Poisson diffeomorphisms, $\mod(M)=0$ implies that  $\mod_G(M)=0$ and hence that $\mod(M//G)=0$.

In the final section of the paper we discuss how the previous results at the level of the symplectic groupoids integrating Poisson manifolds, relating the modular class of Poisson manifolds and the modular characters of their integrating symplectic groupoids.

\vskip 5pt
\noindent
\textbf{Acknowledgments.}
The authors would like to thank the Centre de Recerca Matem\`{a}tica, where part of this work was carried out, for its hospitality and support. We also would like to thank Yvette Kosmann-Schwarzbach for many comments on a preliminary version of this manuscript. Finally, the anonymous referee made many comments and suggestions that helped improved greatly this manuscript.

\tableofcontents

%%%%%%%%%%%%%%%%%%%%%%%%%%%%%%%%%%%%%%%%%%%%%
%%%%%%%%%%%%%%%%%%%%%%%%%%%%%%%%%%%%%%%%%%%%%
%%%%%%%%%%%%%%%%%%%%%%%%%%%%%%%%%%%%%%%%%%%%%
\section{Poisson maps and the modular class}%
\label{sec:Poisson:map}                     %
%%%%%%%%%%%%%%%%%%%%%%%%%%%%%%%%%%%%%%%%%%%%%
%%%%%%%%%%%%%%%%%%%%%%%%%%%%%%%%%%%%%%%%%%%%%
%%%%%%%%%%%%%%%%%%%%%%%%%%%%%%%%%%%%%%%%%%%%%

In this section, we  first review  the notion of modular class for Poisson manifolds,
Lie algebroids and morphisms of Lie algebroids. Then we introduce the modular class of a Poisson map
and we state its basic properties.

%%%%%%%%%%%%%%%%%%%%%%%%%%%%%%%%%%%%%%%%%%%%%%%%%%%%%%%%%%%%%%%%
\subsection{Modular class of a Poisson manifold}               %
\label{subsec:mod:Poisson:manfld}                              %
%%%%%%%%%%%%%%%%%%%%%%%%%%%%%%%%%%%%%%%%%%%%%%%%%%%%%%%%%%%%%%%%

Let $(M,\{\cdot,\cdot\})$ be a Poisson manifold. We will denote by
$\pi\in\X^2(M)$ the associated Poisson tensor which is given by
\[ \pi(\d f,\d  g):=\{f,g\}, \quad (f,g \in  C^\infty(M))\]
and by $\pi^\sharp:T^\ast M \rightarrow TM$ the vector bundle map defined by
\[
\pi^\sharp (\d h)=X_h:=\{h,\cdot\},
\]
where $X_h$ is the hamiltonian vector field determined by $h\in
C^\infty(M)$. Recall also that the Poisson-Lichnerowicz cohomology
\cite{Lichnerowicz} of $(M,\pi)$ is the cohomology of the complex of multivector fields
$(\X^\bullet(M),\d_\pi)$, where the coboundary operator is defined
by taking the Schouten bracket with the Poisson tensor:
\[ \d_{\pi}A\equiv [\pi, A].\]
This cohomology is denoted  by $H^\bullet_\pi(M)$. We will be
mainly interested in the first Poisson cohomology space
$H^1_\pi(M)$, which is just the space of Poisson vector fields
modulo the hamiltonian vector fields. Note that our conventions
are such that the hamiltonian vector field associated with the
function $h$ is given by:
\begin{equation}
\label{eq:hamiltons} X_h=-[\pi,h]=-\d_\pi h.
\end{equation}

If we assume that $M$ is oriented, then we can ask the following question:
\begin{quote}
\emph{Is there a volume form $\mu\in\Omega^{\top}(M)$ which is invariant under all hamiltonian flows?}
\end{quote}
To answer  this question one chooses \emph{any} volume form $\mu$ and computes its Lie derivative along hamiltonian vector fields.
This leads to a unique vector field $X_\mu\in\X(M)$ such that:
\[ \Lie_{X_f}\mu=X_{\mu}(f)\mu.\]
One calls $X_{\mu}$ the \textbf{modular vector field} of the Poisson manifold $(M,\pi)$ relative to $\mu$.
It is easy to check that:
\begin{enumerate}[(a)]
\item The vector field $X_\mu$ is Poisson: $\Lie_{X_\mu}\pi=0$;
\item If $\nu=g\mu$ is another volume form, then:
\begin{equation}
\label{eq:change:form} X_{g\mu}=X_\mu-X_{\ln|g|}.
\end{equation}
\end{enumerate}
This leads to the following definition, which is due to Alan Weinstein \cite{Weinstein1}:
\begin{defn}
The \textbf{modular class} of a Poisson manifold $(M,\pi)$ is the Poisson cohomology class
\[ \mod(M):=[X_\mu]\in H^1_\pi(M).\]
\end{defn}

Note that we can find a volume form $\mu$ with $X_\mu=0$ if and only if $\mod(M)=0$.
Since $X_\mu=0$ if and only if $\mu$ is invariant under all hamiltonian flows, the
modular class is the obstruction to the
existence of a volume form in $(M,\pi)$ invariant under all hamiltonian flows.

\begin{rem}
We have assumed that $M$ is orientable. If this is not the case, one can still define
modular vector fields and modular  classes  if one replaces volume forms by half densities. In this paper, we will always assume the manifolds to be orientable, but our results remain valid in the non-orientable case.
\end{rem}

\begin{ex}
\label{ex:Poisson-Lie}
Let $(G,\pi_G)$ be a Poisson-Lie group with Lie algebra $\gg$. The adjoint characters of $\gg^*$ and $\gg$ denoted, respectively, by $\chi_0\in\gg$ and by $\vartheta_0\in\gg^*$, are defined by:
\[ \langle \eta,\chi_0\rangle=\tr(\ad_{\gg^*}\eta),\quad \langle \vartheta_0,\xi\rangle=\tr(\ad_{\gg}\xi),\quad (\xi\in\gg,\eta\in\gg^*).\]

Evens, Lu and Weinstein (\cite{ELW}) proved that the modular vector field of $\pi_G$ relative to any left invariant volume form $\nu^L$ is given by:
\begin{equation}\label{eq:mod:class:Poisso:Lie:grp}
X_{\nu^L}=\frac{1}{2}\left(\chi_0^L+\chi_0^R+\pi_G(\vartheta_0^L)\right).
\end{equation}
The $1$-form $\vartheta_0^L$ is exact: $\vartheta_0^L=\d\ln f_0$, where $f_0:=\det \Ad_G:\wedge^{\top}\gg \to \wedge^{\top}\gg$. It follows that the modular class of $G$ is
\[\mod G=\left[\frac{1}{2}(\chi_0^L+\chi_0^R)\right]\in H^1_{\pi_G}(G)\]
and hence $\mod G$=0 if and only if $\chi_0=0$, i.e. $\gg^*$ is an unimodular Lie algebra.
\end{ex}

Our next question, which is the main subject of this paper is the following:
\begin{quote}
\emph{Let $\phi:M\to N$ be a Poisson map. Is there a relationship between $\mod(M)$ and $\mod(N)$?}
\end{quote}
Note that a Poisson map, in general, does not induce a map in Poisson cohomology.

\begin{ex}
\label{ex:basic}
Consider $M=\Rr^4$ with the canonical symplectic structure $\omega$, so the associated Poisson structure is:
\[ \pi_M=\frac{\partial}{\partial x}\wedge\frac{\partial}{\partial y}+\frac{\partial}{\partial z}\wedge\frac{\partial}{\partial w}.\]
The volume form $\mu=\omega^2$ is invariant (the Liouville theorem), so $\mod(M)=0$.

On the other hand, consider $N=\Rr^2$ with the linear Poisson structure:
\[ \pi_N=a\frac{\partial}{\partial a}\wedge\frac{\partial}{\partial b}.\]
The modular vector field relative to $\nu=\d a\wedge \d b$ is $X_\nu=\frac{\partial}{\partial b}$.
Since the axis $a=0$ consists of zeros of $\pi_N$, we see that $X_\nu$ cannot be hamiltonian and $\mod(N)\not=0$.

Finally, we observe that there is a Poisson map $\phi:\Rr^4\to\Rr^2$ given by:
\[ \phi(x,y,z,w)=(y,zw-xy). \]
This map is a surjective submersion when restricted to the open subset of $\Rr^4$ obtained by excluding the $x$-axis.
\end{ex}

We will introduce later the \emph{modular class of a Poisson map} which, in a sense to be made precise, measures the failure in preserving the modular classes. For that, we need to turn first to Lie algebroids.

%%%%%%%%%%%%%%%%%%%%%%%%%%%%%%%%%%%%%%%%%%%%%%%%%%%%%%%%%%%%%%%%
\subsection{Modular class  of a Lie algebroid}                 %
\label{subsec:mod:Lie:algbrd}                                  %
%%%%%%%%%%%%%%%%%%%%%%%%%%%%%%%%%%%%%%%%%%%%%%%%%%%%%%%%%%%%%%%%

A Lie algebroid also has a modular class which lives in Lie algebroid cohomology, as
was first explained in \cite{Weinstein1} and then studied in detail in \cite{ELW}.

Let $A\to M$ be a Lie algebroid over $M$, with anchor $\rho:A\to TM$
and Lie bracket $[\cdot,\cdot]:\Gamma(A)\times\Gamma(A)\to\Gamma(A)$. We will denote by
$\Omega^k(A)\equiv\Gamma(\wedge^k A^*)$ the $A$-forms and by $\X^k(A)\equiv\Gamma(\wedge^k A)$
the $A$-multivector fields. Recall that we have a Cartan calculus associated with $A$, where one defines the
$A$-differential $\d_A:\Omega^k(A)\to \Omega^{k+1}(A)$, the Lie $A$-derivative operator $\Lie_X$ and
the contraction operator $i_X$, where $X\in\X(A)$ is an ``A-vector field''. Moreover,
we also have Cartan's magic formula (see, e.g., \cite{ELW}):
\[ \Lie_X=i_X\d_A+\d_A i_X.\]
The cohomology of the complex $(\Omega^\bullet(A),\d_A)$ is the Lie algebroid cohomology which will be
denoted by $H^\bullet(A)$. Also, the space of $A$-multivector fields carries a graded Lie algebra bracket $[~,~]_A$, the $A$-Schouten bracket, which extends the Lie bracket on $\Gamma(A)=\X(A)$. For $X\in\X(A)$ and $P\in\X^k(A)$, we also write the bracket $[X,P]_A$ as $\Lie_X P$.

A representation of $A$ is a vector bundle $E\to M$ together with a flat $A$-connection $\nabla$ (see, e.g.,
\cite{CF0,Fernandes1}). The usual operations $\oplus$ and $\otimes$ on vector bundles turn the space of
representations $\Rep(A)$ into a semiring. Given a morphism of Lie algebroids:
\[
\xymatrix{ A\ar[r]^{\Phi}\ar[d]& B\ar[d]\\
M\ar[r]_{\phi}& N}
\]
there is a pullback operation on representations $E\mapsto \phi^{*}E$, which gives a morphism of
rings $\phi^{*}:\Rep(B)\to\Rep(A)$ (often, we will denote by the same letter the bundle and the
representation).

Representations, as usual, have characteristic classes (see, e.g., \cite{Crainic}). Here we are
interested in line bundles $L\in\Rep(A)$ for which the only characteristic class can be obtained as follows:
Assume first that $L$ is orientable, so that it carries a nowhere vanishing section $\mu\in\Gamma(L)$. Then:
\[ \nabla_X \mu=\langle \al_\mu,X\rangle\mu,\quad \forall X\in\X(A).\]
The 1-form $\al_\mu\in\Omega^1(A)$ is $\d_A$-closed and is called the \textbf{characteristic cocycle} of the
representation $L$. Its cohomology class is independent of the choice of section $\mu$ and
defines the characteristic class of the representation $L$:
\[ \ch(L):=[\al_\mu]\in H^1(A).\]
One checks easily that if $L,L_1,L_2\in\Rep(A)$, then:
\[ \ch(L^*)=-\ch(L),\qquad \ch(L_1\otimes L_2)=\ch(L_1)+\ch(L_2).\]
Also, if $(\Phi,\pi):A\to B$ is a morphism of Lie algebroids, and $L\in\Rep(B)$ then:
\[ \ch(\phi^{*}L)=\Phi^*\ch(L),\]
where $\Phi^*:H^\bullet(B)\to H^\bullet(A)$ is the map induced by $\Phi$ at the level of cohomology.
If $L$ is not orientable, then one defines its characteristic class to be  one half that
of the representation $L\otimes L$, so the formulas above still hold, for non-orientable line bundles.

Every Lie algebroid $A\to M$ has a canonical representation on the
line bundle $L^A=\wedge^{\top}A\otimes\wedge^{\top}T^*M$:
\[ \nabla_X (\omega\otimes\mu)=\Lie_X\omega\otimes\mu+\omega\otimes\Lie_{\rho(X)}\mu. \]

Then we set:
\begin{defn}
The \textbf{modular cocycle} of a Lie algebroid $A$ relative to a nowhere vanishing section
$\omega\otimes\mu\in\Gamma(\wedge^{\top}A\otimes\wedge^{\top}T^*M)$ is the characteristic cocycle
$\al_{\omega\otimes\mu}$ of the representation $L$. The \textbf{modular class} of $A$ is the characteristic class:
\[ \mod(A):=[\al_{\omega\otimes\mu}]\in H^1(A).\]
\end{defn}

\begin{ex}
\label{ex:mod:Poisson}
For any Poisson manifold $(M,\pi)$ there is a natural Lie algebroid
structure on its cotangent bundle $T^*M$: the anchor is
$\rho=\pi^\sharp$ and the Lie bracket on sections of $A=T^*M$, i.e.,
on 1-forms, is given by:
\[ [\al,\be]=\Lie_{\pi^\sharp\al}\be-\Lie_{\pi^\sharp\be}\al-\d\pi(\al,\be).\]
The Poisson cohomology of $(M,\pi)$ is just the Lie algebroid cohomology
of $T^*M$, and the two definitions above of the modular class differ
by a multiplicative factor:
\[ \mod(T^*M)=2\mod(M).\]
Because of this relation, a factor 2 appears in some of our formulas.
\end{ex}

If $\Phi:A\to B$ is a morphism of Lie algebroids covering a map $\phi:M\to N$,
the induced morphism at the level of cohomology $\Phi^*:H^\bullet(B)\to H^\bullet(A)$,
in general, does not map the modular class  of $B$ to that of $A$. Therefore one sets (\cite{KLW}):

\begin{defn}
The \textbf{modular class} of a Lie algebroid morphism $\Phi:A\to B$ is the cohomology
class defined by:
\[ \mod(\Phi):=\mod(A)-\Phi^*\mod(B)\in H^1(A).\]
\end{defn}

Clearly, if $\Phi:A\to B$ and $\Psi:B\to C$ are Lie algebroid morphisms, then:
\[ \mod(\Psi\circ\Phi)=\mod(\Phi)+\Phi^*\mod(\Psi).\]

The modular class of a morphism $(\Phi,\phi):A\to B$ can be seen as the modular class of a representation.
Namely, one takes the canonical representations $L^A\in\Rep(A)$ and $L^B\in\Rep(B)$
and forms the representation
$L^\Phi:=L^A\otimes\phi^{*}(L^B)^*$. Then:

\begin{prop}[\cite{KLW}]
\label{prop:mod:morphism:rep}
Let $(\Phi,\phi):A\to B$ be a Lie algebroid morphism. Then:
\[ \mod(\Phi)=\ch(L^\Phi).\]
\end{prop}

We refer to \cite{KLW} for the proof and for other properties of the modular classes of Lie algebroids and their morphisms.

%%%%%%%%%%%%%%%%%%%%%%%%%%%%%%%%%%%%%%%%%%%%%%%%%%%%%%%%%%%%%%%%
\subsection{Modular class of a Poisson map: definition}        %
\label{subsec:mod:Poisson:map}                                 %
%%%%%%%%%%%%%%%%%%%%%%%%%%%%%%%%%%%%%%%%%%%%%%%%%%%%%%%%%%%%%%%%

Let $\phi:M\to N$ be a Poisson map. Note that, in general, $\phi$ \emph{does
not} induce a morphism neither of the associated cotangent Lie algebroids,
nor in Poisson cohomology. In order to understand what happens in cohomology,
we need(\footnote{A submanifold $C\subset M$ of a Poisson manifold is called \emph{coisotropic} if
$\pi^\sharp(\nu^*(C))\subset TC$, where $\nu^*(C)=(TC)^0$ denotes the conormal
bundle of $C$ in $M$.}):

\begin{prop}[Weinstein \cite{Weinstein2}]
Let $\phi:M\to N$ be a map between two Poisson manifolds. The following statements are equivalent:
\begin{enumerate}[(i)]
\item $\phi$ is a Poisson map;
\item $\graf(\phi)\subset M\times \overline{N}$ is a coisotropic submanifold;
\item The conormal bundle $\nu^*(\graf(\phi))$ is a Lie subalgebroid of $T^*M\times T^*\overline{N}$.
\end{enumerate}
\end{prop}

Here we are using the following standard convention: if $(N,\pi_N)$ is a Poisson manifold, then
$\overline{N}$ denotes the opposite Poisson manifold $(N,-\pi_N)$. So if $(M,\pi_M)$
is also a Poisson manifold, then $M\times\overline{N}$ denotes the cartesian product of $M$
and $N$ with the Poisson tensor $\pi_M\times(-\pi_N)$.

The following proposition provides a more convenient description of the Lie algebroid $\nu^*(\graf(\phi))$:

\begin{prop}
\label{prop:pull:back:algbrd}
Let $\phi:M\to N$ be a Poisson map. Then $A:=\phi^{*}T^*N$ carries a natural Lie algebroid
structure which is characterized by:
\begin{align}
\label{eq:bracket:pull:back:algbrd}
[\phi^{*}\al,\phi^{*}\be]_A&=\phi^{*}[\al,\be]_{\pi_N},\\
\label{eq:anchor:pull:back:algbrd}
\rho_A(\phi^{*}\al)&=\pi_M^\sharp(\phi^*\al),
\end{align}
for $\al,\be\in\Omega^1(N)$. For this structure, the natural maps
\begin{equation}
\label{diag:mod:class:poisson}
\xymatrix{
\phi^{*}T^*N\ar[r]^i\ar[dr]_{j}& T^*M\ar@{-->}[d]^{\Phi}\\
&T^*N}
\end{equation}
are Lie algebroid morphisms and we have $\phi^{*}T^*N\simeq \nu^*(\graf(\phi))$.
\end{prop}

\begin{proof}
We have an isomorphism of vector bundles:
\[
\xymatrix{
\phi^{*}T^*N\ar[r]\ar[d]& \nu^*(\graf(\phi))\subset T^*M\times T^*\overline{N}\ar@<-6ex>[d]\ar@<7ex>[d]\\
M\ar[r] &\graf(\phi))\subset M\times \overline{N}}
\]
given by $(m,\al)\mapsto ((\d_m\phi)^*\al,-\al)$ and covering the diffeomorphism $m\mapsto(m,\phi(m))$.
Under this isomorphism, the Lie bracket and the anchor on $\nu^*(\graf(\phi))$ lead to the expressions
\eqref{eq:bracket:pull:back:algbrd} and \eqref{eq:anchor:pull:back:algbrd}, while composition with the
projections on each factor give Lie algebroid morphisms.

In general, the natural map $\phi^{*}T^*N\to T^*M$ will not be injective, so $\phi^{*}T^*N$
will not be a Lie subalgebroid of $T^*M$. Still, sections of $\phi^{*}T^*N$ of the form $\phi^{*}\al$
where $\al\in\Omega^1(N)$, can be identified with the sections $\phi^*\al\in \Omega^1(M)$
(the usual pullback of 1-forms). An arbitrary section of $\phi^{*}T^*N$ is just a 1-form along $\phi$
and can be written as a finite combination $\sum_i f_i\phi^{*}\al_i$, where $f_i\in C^\infty(M)$ and
$\al_i\in\Omega^1(N)$. Hence, expressions \eqref{eq:bracket:pull:back:algbrd} and
\eqref{eq:anchor:pull:back:algbrd} determine completely the Lie bracket and the anchor of $A=\phi^{*}T^*N$.
\end{proof}

Henceforth, for a Poisson map  $\phi:M\to N$, we will denote by $\X^\bullet(\phi)$ the complex of the Lie algebroid $\phi^*T^*N$ and by $H^\bullet_\pi(\phi)$ its cohomology. Our basic definition is the following:

\begin{defn}
\label{defn:mod:map}
Let $\phi:M\to N$ be a Poisson map. The \textbf{modular class} of $\phi$ is the class:
\[ \mod(\phi):=i^*\mod(M)-j^*\mod(N)\in H^1_\pi(\phi). \]
\end{defn}

In order to motivate this definition, suppose that the map $\phi:M\to N$ was covered by a Lie algebroid
homomorphism $\Phi:T^*M\to T^*N$  (the dotted arrow in the
diagram), making the diagram \eqref{diag:mod:class:poisson} commutative.
The composition property for the modular classes  of Lie algebroid
morphisms yields:
\[ \mod(j)=\mod(i)+i^*\mod(\Phi).\]
Hence, we would have:
\begin{align*}
i^*\mod(\Phi)&=\mod(j)-\mod(i)\\
						 &=\left(\mod(\phi^{*}T^*N)-j^*\mod(T^*N)\right)-\left(\mod(\phi^{*}T^*N)-i^*\mod(T^*M)\right)\\
						 &=i^*\mod(T^*M)-j^*\mod(T^*N)
\end{align*}
and this justifies our definition.

%%%%%%%%%%%%%%%%%%%%%%%%%%%%%%%%%%%%%%%%%%%%%%%%%%%%%%%%%%%%%%%%
\subsection{Modular class of a Poisson map: basic properties}  %
\label{subsec:mod:Poisson:propert}                             %
%%%%%%%%%%%%%%%%%%%%%%%%%%%%%%%%%%%%%%%%%%%%%%%%%%%%%%%%%%%%%%%%

Notice that an element $P\in\Omega^k(\phi^{*}T^*N)$ is just a $k$-vector field along $\phi$, i.e., a smooth
map $P:M\to \wedge^kTN$ such that $P_x\in\wedge^kT_{\phi(x)}N$. Therefore, the modular class of $\phi$
has representatives which are vector fields along $\phi$. Note also, that the maps
$i^*:H^\bullet_\pi(M)\to H^\bullet_\pi(\phi)$ and $j^*:H^\bullet_\pi(\phi)\to H^\bullet_\pi(N)$
at the level of multivector fields are just:
\begin{align*}
&i^*:\X^\bullet(M)\to \X^\bullet(\phi),\ P\mapsto \d\phi\cdot P,\\
&j^*:\X^\bullet(\phi)\to \X^\bullet(N),\ Q\mapsto P\circ\phi.
\end{align*}
The following lemma, which gives an explicit representative for the modular class, should now be obvious:

\begin{lem}
Let $\phi:M\to N$ be a Poisson map and fix volume forms $\mu\in\Omega^{\top}(M)$, $\nu\in\Omega^{\top}(N)$.
The modular class $\mod(\phi)$ is represented by:
\[ X_{\mu,\nu}=\d\phi\cdot X_{\mu}-X_{\nu}\circ\phi,\]
where $X_{\mu}$ and $X_{\nu}$ are the modular vector fields of $\pi_M$ and $\pi_N$ relative to $\mu$ and $\nu$.
\end{lem}

We will refer to $X_{\mu,\nu}$ as the \textbf{modular vector field} of $\phi$ relative to $\mu$ and $\nu$.

\begin{cor}
The class $\mod(\phi)$ is the obstruction to the existence of volume forms $\mu\in\Omega^{\top}(M)$ and $\nu\in\Omega^{\top}(N)$,
such that the modular vector fields $X_\mu$ and $X_\nu$ are $\phi$-related.
\end{cor}

\begin{proof}
The Poisson map $\phi$ has trivial modular class if
  its modular vector fields  are exact in the Lie algebroid cohomology of $\phi^*T^*N$, i.e., if for each $\mu\in\Omega^{\top}(M)$ and $\nu\in\Omega^{\top}(N)$,
 \[X_{\mu,\nu}=\d_{\phi^*T^*N} f=-\d \phi \cdot X_f, \quad \mbox{ for some $f\in C^\infty(M)$}.\]
By definition $X_{\mu,\nu}=\d\phi \cdot X_\mu-X_\nu\smalcirc \phi$, hence we have $\d\phi.(X_\mu+X_f)=X_\nu\smalcirc \phi$, and taking into account equation (\ref{eq:change:form}), we conclude that
 $X_{e^{-f}\mu}$ and $X_\nu$ are $\phi$-related.
\end{proof}

\begin{cor}
The modular class of a Poisson diffeomorphism is trivial.
\end{cor}

\begin{proof}
Suppose $\phi:M\to N$ is a Poisson diffeomorphism and let $\nu$ be a volume form on $N$. The modular vector field of $\phi$ relative to $\phi^*\nu$ and $\nu$ vanishes:
\begin{align*}
\d \phi\cdot X_{\phi^*\nu}(f)\,\phi^*\nu & =X_{\phi^*\nu}(f\smalcirc\phi)\,\phi^*\nu= \Lie_{\pi_M^\sharp(\phi^*\d f)}\phi^*\nu\\
&=\phi^*\Lie_{\pi_N^\sharp(\d f)}\nu=(X_\nu(f)\smalcirc \phi)\, \phi^*\nu.
\end{align*}
\end{proof}

Our next proposition gives the composition property of the modular
classes of Poisson maps. For that, we
note that if $\phi:M\to N$ and $\psi:N\to Q$ are Poisson maps, there are induced maps in cohomology:
\begin{align*}
&\psi_*:H^\bullet_\pi(\phi)\to H^\bullet_\pi(\psi\circ\phi),\ P\mapsto \d\psi\cdot P ,\\
&\phi^*:H^\bullet_\pi(\psi)\to H^\bullet_\pi(\psi\circ\phi),\ P\mapsto P\circ\phi .
\end{align*}

\begin{prop}\label{prop:comp:poisson:map}
Let $\phi:M\to N$ and $\psi:N\to Q$ be Poisson maps. Then:
\[ \mod(\psi\circ\phi)=\psi_*\mod(\phi)+\phi^*\mod(\psi). \]
\end{prop}

\begin{proof}
The following diagram commutes:
\[
\xymatrix{
& & H^\bullet_\pi(N)\ar[dl]_{j_\phi^*}\ar[dr]^{i_\psi^*}\\
H^\bullet_\pi(M)\ar[r]^{i_\phi^*}\ar[drr]_{i_{\psi\circ\phi}^*}& H^\bullet_\pi(\phi)\ar[dr]^{\psi_*}
& &H^\bullet_\pi(\psi)\ar[dl]_{\phi^*}&H^\bullet_\pi(Q)\ar[l]_{j_\psi^*}\ar[dll]^{j_{\psi\circ\phi}^*}\\
& & H^\bullet_\pi(\psi\circ\phi)}
\]
Hence, we find:
\begin{align*}
\mod(\psi\circ\phi)&=i_{\psi\circ\phi}^*\mod(M)-j_{\psi\circ\phi}^*\mod(Q)\\
&=\psi_*i_\phi^*\mod(M)-\phi^*j_\psi^*\mod(Q)\\
&=\psi_*i_\phi^*\mod(M)-\psi_*j_\phi^*\mod(N)+\psi_*j_\phi^*\mod(N)-\phi^*j_\psi^*\mod(Q)\\
&=\psi_*(i_\phi^*\mod(M)-j_\phi^*\mod(N))+\phi^*(i_\psi^*\mod(N)-j_\psi^*\mod(Q))\\
&=\psi_*\mod(\phi)+\phi^*\mod(\psi) .
\end{align*}
\end{proof}

Note that if $\mu$, $\nu$ and $\lambda$ are volume forms on $M$, $N$ and $Q$, respectively, then
the previous proposition can be stated in terms of the corresponding modular vector fields as:
\[ X_{\mu,\lambda}=\d\psi\cdot X_{\mu,\nu}+X_{\nu,\lambda}\smalcirc\phi.\]

Finally, we show that the modular class of a Poisson map can be seen as the characteristic class
of a representation, in analogy with Proposition \ref{prop:mod:morphism:rep}.

\begin{prop}
\label{prop:mod:poisson:rep}
Let $\phi:M\to N$ be a Poisson map. There is a natural representation of $\phi^*T^*N$ on the
line bundle $L^\phi:=\otimes^2\wedge^\top T^*M \otimes^2\wedge^\top \phi^*TN$, and we have:
\[ \mod(\phi)=\frac{1}{2}\ch(L^\phi).\]
\end{prop}

\begin{proof}
We define a representation of $\phi^*T^*N$ on the line bundle $\otimes^2\wedge^\top T^*M$ by setting:
\[ \nabla_{\phi^*\al}(\mu\otimes\nu):=[\phi^*\al,\mu]_{\pi_M}\otimes\nu+\mu\otimes\Lie_{\pi^\sharp_M\phi^*\al}\nu.\]
and another representation on $\otimes^2\wedge^\top \phi^*T^*N$ by setting:
\[ \nabla_{\phi^*\al}(\phi^*\mu\otimes\phi^*\nu):=\phi^*[\al,\mu]_{\pi_N}\otimes\phi^*\nu+
\phi^*\mu\otimes\phi^*\Lie_{\pi^\sharp_N \al}\nu.\]
The tensor product of the first representation with the dual of the second representation defines a
representation of $\phi^*T^*N$ on the line bundle
\[ L^\phi:=\otimes^2\wedge^\top T^*M \otimes^2\wedge^\top \phi^*TN.\]
The rest follows from the definitions taking into account Example \ref{ex:mod:Poisson}.
\end{proof}

%%%%%%%%%%%%%%%%%%%%%%%%%%%%%%%%%%%%%%%%%%%%%%%%%%%%%%%%%%%%%
%%%%%%%%%%%%%%%%%%%%%%%%%%%%%%%%%%%%%%%%%%%%%%%%%%%%%%%%%%%%%
%%%%%%%%%%%%%%%%%%%%%%%%%%%%%%%%%%%%%%%%%%%%%%%%%%%%%%%%%%%%%
\section{Poisson immersions and submanifolds}               %
\label{sec:Poisson:submanfld}                               %
%%%%%%%%%%%%%%%%%%%%%%%%%%%%%%%%%%%%%%%%%%%%%%%%%%%%%%%%%%%%%
%%%%%%%%%%%%%%%%%%%%%%%%%%%%%%%%%%%%%%%%%%%%%%%%%%%%%%%%%%%%%
%%%%%%%%%%%%%%%%%%%%%%%%%%%%%%%%%%%%%%%%%%%%%%%%%%%%%%%%%%%%%

In this section we study the modular class of a Poisson immersion.

%%%%%%%%%%%%%%%%%%%%%%%%%%%%%%%%%%%%%%%%%%%%%%%%%%%%%%%%%%%%%%%%
\subsection{The relative modular class}                        %
\label{subsec:Poisson:relative}                                %
%%%%%%%%%%%%%%%%%%%%%%%%%%%%%%%%%%%%%%%%%%%%%%%%%%%%%%%%%%%%%%%%

Let $(M,\pi)$ be a Poisson manifold and let $N\subset M$ be a Poisson submanifold. This means that the submanifold
$N$ carries a Poisson structure $\pi_N$ and the inclusion $\phi:N\hookrightarrow M$ is a Poisson map. Notice that $\phi^*T^*M=T^*_N M$ is just the restricted cotangent Lie algebroid of $M$ along $N$. The relevant cohomology is then the \textbf{restricted Poisson cohomology} $H^\bullet_{\pi,N}(M):=H^\bullet(\phi)$, introduced by Ginzburg and Lu in \cite{GL}. The corresponding complex can be described as follows: the cochains
are the multivector fields along $N$, which we denote by $\X^\bullet_N(M):=\Gamma(\wedge^\bullet T_N M)$, while
the differential is given by:
\[ \d_\pi P=[\widetilde{P},\pi]|_N,\]
where $\widetilde{P}\in\X^\bullet(M)$ denotes any (local) extension of $P\in\X^\bullet_N(M)$ to $M$(\footnote{Ginzburg and Lu use the term
\emph{relative cohomology} instead of restricted cohomology. We believe that
their term should be reserved for the cohomology $H^\bullet_\pi(M,N)$ of the quotient complex $\X^\bullet(M,N):=\X^\bullet_N(M)/\X^\bullet(N)$,
associated with (the dual of) the anchor $\pi^\sharp:T^*_N M\to TN$ (see \cite{CF1}).}).

\begin{defn}
The \textbf{relative modular class} of a Poisson submanifold $\phi:N\hookrightarrow M$ is defined as minus the
modular class of $\phi$, denoted\footnote{Note that in this section the role of $N$ and $M$ is switched from the previous section. This will also be the case in the following sections whenever we deal with submanifolds.}:
\[ \mod_N(M):=-\mod(\phi)\in H^1_{\pi,N}(M).\]
\end{defn}

The reason for the minus sign will became clear from the discussion that follows. In particular, the relative modular class $\mod_N(M)$ lies in the first restricted Poisson cohomology $H^1_{\pi,N}(M)$, which is the quotient of the Poisson vector fields along $N$ modulo the hamiltonian
vector fields on $N$.

\begin{rem}
\label{rem:relative:class}
A representative of this class can be obtained by choosing volume forms $\mu\in\Omega^\top(M)$ and $\nu\in\Omega^\top(N)$. Then $\mod_N(M)$ is represented by the vector field along $N$ given by:
\[ X_{\mu,\nu}={X_\mu}|_N-X_\nu. \]
This representative depends only on the values of $\mu$ along $N$ and on $\nu$. Of course, its class is independent of these choices.
\end{rem}

The following elementary properties of the relative modular class follow directly from its definition:

\begin{prop}
\label{prop:basic:relative}
Let $(M,\pi)$ be a Poisson manifold and let $N\subset M$ be a Poisson submanifold. Then:
\begin{enumerate}[(i)]
\item If $\mod_N(M)=0$, then the modular vector field of $M$ relative to any volume form
$\mu\in\Omega^\top(M)$ is tangent to $N$.
\item If $\pi_N=0$, then $X_\mu|_N$ is independent of the volume form $\mu\in\Omega^\top(M)$ and represents $\mod_N(M)$.
\item If $\mod_N(M)=\mod(M)=0$, then $\mod(N)=0$.
\end{enumerate}
\end{prop}

\begin{ex}
For any open subset $O\hookrightarrow M$ the relative modular class vanishes: $\mod_O(M)=0$.
\end{ex}

\begin{ex}
\label{ex:sympl:leaf:R4}
A symplectic leaf can have a non-trivial relative modular class. On $M=\Rr^4$, with coordinates
$(x,y,z,w)$, consider the Poisson structure:
\[ \pi=x\frac{\partial}{\partial x}\wedge \frac{\partial}{\partial y}+\frac{\partial}{\partial z}\wedge\frac{\partial}{\partial w}.\]
The modular vector field relative to the volume form $\mu=\d x\wedge\d y\wedge\d z\wedge\d w$ is given by:
\[ X_\mu=\frac{\partial}{\partial y}. \]
On the other hand, the symplectic leaves of $\pi$ are:
\begin{itemize}
\item The two open 4-dimensional leaves $x>0$ and $x<0$;
\item The 1-parameter family of 2-dimensional leaves $(x,y)=(0,c)$, where $c\in\Rr$.
\end{itemize}
If $S$ is a 2-dimensional symplectic leaf, then $X_\mu$ is not tangent to $S$ so we conclude that $\mod_S(M)\not=0$.
\end{ex}

\begin{ex}
Consider a topologically stable Poisson structure on a compact oriented surface $M=\Sigma$ (see \cite{Radko}).
This is a Poisson structure which vanishes linearly along a set of smooth disjoint closed curves
$Z=\bigcup_{i}\gamma_i$. The (trivial) Poisson submanifold $Z\subset \Sigma$ has no non-trivial
hamiltonian vector fields, so its relative modular class $\mod_Z(\Sigma)$ is represented by a \emph{canonical}
Poisson vector field along $Z$.

In order  to determine  this vector field, choose some volume form $\mu$ on $\Sigma$. The corresponding modular vector field
$X_\mu$ is a Poisson vector field on $\Sigma$, so its flow maps symplectic leaves to symplectic leaves. This
means that $X_\mu$ must be tangent to $Z$ and we find that the relative modular class of $Z$ is given by
\[ \mod_Z(\Sigma)=X_\mu|_Z.\]
In other words, all modular vector fields $X_\mu$ are tangent to $Z$ and induce the same vector field $\mod_Z(\Sigma)$
on $Z$. The periods of the vector field $\mod_Z(\Sigma)$ along the curves $\gamma_i$ give global invariants which were
used in \cite{Radko} to classify the topologically stable Poisson
structures  on surfaces.
\end{ex}

%%%%%%%%%%%%%%%%%%%%%%%%%%%%%%%%%%%%%%%%%%%%%%%%%%%%%%%%%%%%%%%%
\subsection{Linear holonomy and the relative modular class}    %
\label{subsec:holonomy}                                        %
%%%%%%%%%%%%%%%%%%%%%%%%%%%%%%%%%%%%%%%%%%%%%%%%%%%%%%%%%%%%%%%%

In order to compute the relative modular class $\mod_N(M)$ of a Poisson submanifold, we observe that the
identification of the conormal bundle as $\nu^*(N)=(TN)^0\subset T^*_N M$ gives a natural representation of the Lie algebroid
$A=T^*_N M$ on the conormal bundle $\nu^*(N)$:
\[ \nabla_\al \be:=[\al,\be], \quad \al\in\Gamma(T^*_N M), \be\in\Gamma(\nu^*(N)).\]
We will denote the dual representation on the normal bundle $\nu(N)$ by the same letter. The line bundle $\wedge^\top\nu(N)$
then becomes a representation of $T^*_N M$, and we have:

\begin{thm}
\label{thm:relative:mod:class}
The relative modular class of a Poisson submanifold $\phi:N\hookrightarrow M$ is given by:
\[ \mod_N(M)=\ch(\wedge^\top\nu(N)).\]
\end{thm}

\begin{proof}
The short exact sequence
\[ 0\rmap \nu^*(N) \hookrightarrow T^*_N M\rmap T^*N\rmap 0 \]
yields the canonical line bundle isomorphism
\[\wedge^\top \nu^*(N) \otimes \wedge^\top T^*N\cong  \wedge^\top T^*_N M. \]
%given by
%$
%(\nu,\mu_S)\rmap \nu\wedge \widetilde{\mu_S}$, where $i(\widetilde{\mu_S})=\mu_S$.

Under this isomorphism, one can identify the representation $L^\phi$ given in
Proposition \ref{prop:mod:poisson:rep} with $\otimes^2\wedge^\top \nu (N)$ defined above.
So, $\mod \phi=\frac{1}{2}\ch(\otimes^2\wedge^\top \nu(N))=\ch(\wedge^\top\nu(N))$.
\end{proof}

The representation of $T^*_N M$ on the normal bundle $\nu(N)$ is well-known: it is called the \textbf{linear holonomy
representation} of $N$ and it plays an important role in several problems in global Poisson geometry (see, e.g., \cite{GG})
In particular, given a cotangent path $a:I=[0,1]\to T^*_N M$ joining $x$ to $y$ in $N$, parallel transport
relative to $\nabla$ along $a$ defines the \textbf{linear Poisson holonomy} map $h_N(a):\nu(N)_x\to\nu(N)_y$
(this is a slight generalization to Poisson submanifolds of the usual definition for symplectic leaves, given
in \cite{Fernandes2,GG}). Then:

\begin{cor}
Let $\phi:N\hookrightarrow M$ be a Poisson submanifold. The class $\mod_N(M)$ is the obstruction to
the existence of a transverse volume form to $N$ invariant under linear holonomy.
\end{cor}

This corollary is closely related to a result of Ginzburg and Golubev for symplectic leaves (see \cite[Theorem 3.5]{GG})
 which admits a generalization to any Poisson submanifold, as we now explain.

Given a cotangent path $a:I\to T^*M$ on a Poisson manifold $(M,\pi)$, covering a path $\gamma:I\to M$, and
a Poisson vector field $X\in\X(M)$ one defines the integral:
\[ \int_a X:=\int_0^1 \langle X|_{\gamma(t)},a(t)\rangle ~ \d t.\]
This integral was introduced in \cite{GG} and further studied in \cite{CF2}, where the following properties are proved:
\begin{enumerate}[(a)]
\item If $X=X_h$ is a hamiltonian vector field, then $\int_a X_h=h(\gamma(1))-h(\gamma(0))$ only depends on the end points.
\item If $a_0$ and $a_1$ are cotangent homotopic, then $\int_{a_0}X=\int_{a_1}X$;
\end{enumerate}
In particular, for a cotangent loop $a$, the integral $\int_a X$ only depends on the Poisson cohomology class $[X]\in H^1_\pi(M)$
and on the cotangent homotopy class $[a]$.

If $\phi:N\hookrightarrow M$ is a Poisson submanifold, notice that the integral $\int_a X$ still makes sense for
a vector field $X\in\X_N(M)$ along $N$, provided $a:I\to T^*_N M$ is a cotangent path with base path lying in the submanifold $N$. Also,
for a cotangent loop $a$, it depends only on the restricted Poisson cohomology class $[X]\in H^1_{N,\pi}(M)$.

\begin{thm}
For any cotangent loop $a:I\to T^*_N M$ on a Poisson submanifold $ \phi:N\hookrightarrow M$ one has:
\begin{equation}
\label{eq:holonomy:mod}
\det h_N(a)=\exp\left(\int_a \mod_N(M)\right).
\end{equation}
\end{thm}

\begin{proof}
Note that when $N$ is a symplectic leaf $S$ of $M$, formula \eqref{eq:holonomy:mod} reduces to:
\begin{equation}
\label{eq:holonomy:mod:symp}
\det h_S(a)=\exp\left(\int_a \mod(M)\right),
\end{equation}
which is the result of Ginzburg and Golubev (\cite[Theorem 3.5]{GG}). We claim that the theorem can be reduced to this case.

Let $a:I\to T^*_N M$ be a cotangent loop covering a loop $\gamma:I\to N$ with base point $x$. Let $S$ be the symplectic leaf containing $\gamma$. Then we have
$S\subset N\subset M$ and we can look at the following linear holonomies:
\begin{itemize}
\item The linear holonomy of the loop $a:I\to T^*M$ relative to the symplectic leaf $S$ of $M$, i.e., $h_S^M(a):(T_xM/T_xS)\to(T_xM/T_xS)$;
\item The linear holonomy of the loop $\phi^*a:I\to T^*N$ relative to the symplectic leaf $S$ of $N$, i.e., $h_S^N(\phi^*a):(T_xN/T_xS)\to(T_xN/T_xS)$
\item The linear holonomy of the loop $a:I\to T^*_NM$ relative to the submanifold $N$ of $M$, i.e., $h_N(a)=h_N^M(a):(T_xM/T_xN)\to(T_xM/T_xN)$.
\end{itemize}
It is clear from the definitions that:
\[ \det h_S^M(a)=\det(h_S^N(\phi^*a))\cdot \det(h_N^M(a)).\]
Hence, it follows from the case of symplectic leaves \eqref{eq:holonomy:mod:symp}, that for any volume forms $\mu\in\Omega^\top(M)$ and $\nu\in\Omega^\top(N)$ we have:
\begin{align*}
\det(h_N(a))&=\det(h_N^M(a))\\
&=\frac{\det h_S^M(a)}{\det(h_S^N(\phi^*a))}\\
&=\exp\left(\int_a\mod(M)-\int_{\phi^*a}\mod(N)\right)\\
&=\exp\left(\int_0^1 \langle {X_\mu}|_{\phi\circ\gamma(t)},a(t)\rangle-\langle {X_\nu}|_{\gamma(t)},\phi^*a(t)\rangle ~ \d t\right)\\
&=\exp\left(\int_0^1 \langle (X_\mu\circ\phi-\d\phi\cdot X_\nu)|_{\gamma(t)},a(t)\rangle ~ \d t\right)\\
&=\exp\left(-\int_a\mod(\phi)\right)=\exp\left(\int_a\mod_M(M)\right),
\end{align*}
as claimed.
\end{proof}

%Let $\gamma:I:\to N$ be the base path of the cotangent loop $a:I\to T^*_S N$ with base point $x=\gamma(0)=\gamma(1)$. Consider $\tilde a_t$ a time-dependent closed
%$1$-form on $M$ such that
%$\ds \tilde a_t(\gamma(t))= a(t)$,
%and the Poisson vector field $X_t=\pi^\sharp(\tilde a_t)$. Linearization of the flow of $X_t$ along $\gamma$ defines a map
%$$\Phi: T_{x}M\to T_{x}M$$
%that preserves tangency to $S$ because $S$ is a Poisson submanifold of $N$ and $X_{t|S}=\phi_*\pi_S^\sharp(\phi^*\tilde a_t)$ is tangential to $S$. So $\Phi_{|TS}$
%encodes  the linear holonomy of $S$ along $\phi^* a$ and $\det \Phi_{|TS}=\det h^S(\phi^*a)=\exp\int_{\phi^* a}\mod S$. Also $\Phi$ defines the linear Poisson
%holonomy  $h_S(a):\nu(S)_x\to\nu(S)_x $  and $\det \Phi=\det h^S(\phi^*a)\cdot \det h_S(a)$.
%
%Let $\mu$ be a volume form on $N$ defined by $\mu_S$ a volume form on $S$ and $\nu$ a top-section of $\nu^*(S)$.
%With respect to these volume forms, we have
% $\det h^N(a)=\det \Phi=h^S(\phi^*a)\cdot \det h_S(a)$. Notice that
%\begin{align*}
%\det h^S(\phi^* a)&=\int_{\phi^*a} X_{\mu_S}=-\int_{I} \eval{\phi^*a(\gamma(t)), X_{\mu_S}(\gamma(t))}\d t\\
%&=-\int_{I} \eval{a(\gamma(t)), \d \phi\cdot X_{\mu_S}(\gamma(t))}=\int_a \d\phi\cdot X_{\mu_S},
%\end{align*}
%hence
%\[
%\det h_S(a)= \frac{\det h^N(a)}{\det h^S(\phi^* a)}=\exp(\int_a X_\mu-\d\phi\cdot X_{\mu_S})=\exp(-\int_a X_{\mu_S,\mu}).
%\]

\begin{rem}
There is a more general formula which is valid for any cotangent path
(not just cotangent loops), as in the case of symplectic leaves (see \cite[Proposition 3.7]{GG}). One fixes volume forms $\mu\in\Omega^\top(M)$ and $\nu\in\Omega^\top(N)$ so there is an induced volume on any normal space $\nu(N)_x$. Then, for any cotangent path $a:I\to T^*_N M$ joining $x\in N$ to $y\in N$ one can show that:
\begin{equation}
\label{eq:holonomy:mod:any}
\det h_N(a)=\exp\left(-\int_a X_{\nu,\mu}\right).
\end{equation}
where the determinant of the linear holonomy $h_N(a):\nu(N)_x\to\nu(N)_y$ is relative to the induced volumes on $\nu(N)_x$ and $\nu(N)_y$.
\end{rem}

%%%%%%%%%%%%%%%%%%%%%%%%%%%%%%%%%%%%%%%%%%%%%%%%%%%%%%%%%%%%%%%%
\subsection{Special Poisson submanifolds}                          %
\label{subsec:Lie-Poisson}                                     %
%%%%%%%%%%%%%%%%%%%%%%%%%%%%%%%%%%%%%%%%%%%%%%%%%%%%%%%%%%%%%%%%

For any Poisson submanifold $\phi:N\hookrightarrow M$ we have the short exact sequence of Lie algebroids:
\begin{equation}
\label{eq:seq:Lie:Poisson}
\xymatrix{0\ar[r] & \nu^*(N) \ar[r]& T^*_N M \ar[r]^{\phi^*} & T^*N\ar[r]& 0.}
\end{equation}
This sequence gives rise to a long exact sequence in cohomology, which includes the terms:
\[ \xymatrix{\dots \ar[r] & H^1_\pi(N) \ar[r]^{\phi_*}& H^1_{N,\pi}(M) \ar[r] & H^1(\nu^*(N)) \ar[r]& H^2_\pi(N) \ar[r]& \dots} \]
The inclusion $\nu^*(N)\to T^*_NM$ pulls back the representation $\wedge^\top\nu(M)$ of $T^*_NM$, which by Proposition \ref{thm:relative:mod:class} computes the relative modular class, to the canonical representation of the bundle of Lie algebras $\nu^*(N)$ on $\wedge^\top\nu(M)$ computing the modular class of $\nu^*(N)$. In other words, in this long exact sequence the relative modular class $\mod_N(M)\in H^1_{N,\pi}(M)$ is mapped to $\mod(\nu^*(N))\in H^1(\nu^*(N))$.

Now notice that $\mod_N(M)$ is mapped to zero if and only if $\mod(M)$ is tangent to $N$, i.e., if $\mod(M)$ admits a representative which is tangent to $N$ (and then it follows that all representatives will be tangent to $M$).This gives immediately that:

\begin{cor}\label{cor:Poisson:sub}
Let $N$ be a Poisson submanifold of $M$. Then $\mod(M)$ is tangent to $N$ if and only if $\mod(\nu^*(N))=0$.
\end{cor}

Let us recall that one calls $N$ a \textbf{Lie-Poisson submanifold} if the sequence \eqref{eq:seq:Lie:Poisson} splits in the category of Lie algebroids, i.e.,
if there exists a Lie algebroid morphism $p:T^*N\to T_N^*M$ such that $\phi^*\circ p=\text{id}$ (see \cite{Xu}). Such submanifolds behave better than general submanifolds in what concerns integrability (\cite{CF2,Xu}). They arise, for example, as fixed point sets of Poisson actions (\cite{FOR,Xu}). A splitting of the sequence \eqref{eq:seq:Lie:Poisson} gives as an isomorphism in cohomology:
\begin{equation}
\label{eq:split:cohomology}
H^1_{N,\pi}(M)\simeq H^1_\pi(N)\oplus H^1(\nu^*(N)).
\end{equation}
We already observed that the second component of $\mod_N(M)$ in this decomposition is $\mod(\nu^*(N))$. The first component is also simple to describe. In fact, a splitting $p:T^*N\to T_N^*M$ of the sequence \eqref{eq:seq:Lie:Poisson}, pulls back the representation $\wedge^\top\nu(M)$ of $T^*_NM$ to a representation of $T^*N$. Hence, $\wedge^\top\nu(M)$ also becomes a representation of $T^*N$ with characteristic class
\[ \ch_{T^*N}(\wedge^\top\nu(N))\in H^1_\pi(N),\]
and we conclude that:

\begin{prop}\label{prop:Lie:Poisson}
Let $N$ be a Lie-Poisson submanifold of $M$. Then under the isomorphism \eqref{eq:split:cohomology} we have:
\[ \mod_N(M)=( \ch_{T^*N}(\wedge^\top\nu(N)), \mod (\nu^*(N))).\]
\end{prop}

A special kind of Lie-Poisson arises when one considers the fixed point set of a Poisson involution: $\Phi:M\to M$, $\Phi\circ\Phi=I$. In this case, the fixed point set
\[ N=\{x\in M: \Phi(x)=x\}\]
has a natural induced Poisson structure but may fail to be a Poisson submanifold (see \cite{FOR,Xu}). The following corollary concerns the case where the fixed point set is also a Poisson submanifold:

\begin{cor}
Let $N$ be a Poisson submanifold of $M$ which is the fixed point set of a Poisson involution $\Phi:N\to N$. Then $\mod(M)$ is tangent to $N$.
\end{cor}

\begin{proof}
We only need to observe that $\mod(\nu^*(N))=0$. In fact, we will show that $\nu^*(N)$ is a bundle of abelian Lie algebras so the result follows.

To prove that $\nu^*(N)$ is abelian note that the Poisson involution $\Phi$ induces an involutive Lie algebroid automorphism of $T^*_NM$, for which $\nu^*(N)=(TN)^0$ coincides with the $(-1)$-eigenspace. Therefore, if $\al$ and $\be$ are local sections of $\nu^*(N)$ then their Lie bracket $[\al,\be]$ must lie both in the $(+1)$-eigenspace and in the $(-1)$-eigenspace, hence it must vanish.
\end{proof}

\begin{ex}
%\label{ex:sympl:leaf:R4}
Consider the Poisson structure on $M=\Rr^4$ given in Example \ref{ex:sympl:leaf:R4}.

On the one hand, the Poisson submanifold $N=\{x=0\}$ is the fixed point set of the Poisson involution defined by $(x,y,z,w)\mapsto (-x,y,z,w)$. Hence, $N$ is a Lie-Poisson submanifold.

On the other hand, note that $M$ is foliated by the 2-dimensional symplectic leaves $S=\{x=0,y=c\}$. As we saw in Example \ref{ex:sympl:leaf:R4},
$\mod(M)$ is not tangent to $S$, so the corollary shows that such a symplectic leaf can not be the fixed point set of a Poisson involution.
\end{ex}

%%%%%%%%%%%%%%%%%%%%%%%%%%%%%%%%%%%%%%%%%%%%%%%%%%%%%%%%%%%%%%%%
%%%%%%%%%%%%%%%%%%%%%%%%%%%%%%%%%%%%%%%%%%%%%%%%%%%%%%%%%%%%%%%%
%%%%%%%%%%%%%%%%%%%%%%%%%%%%%%%%%%%%%%%%%%%%%%%%%%%%%%%%%%%%%%%%
\section{Poisson submersions and reduction}                    %
\label{sec:reduction}				                           %
%%%%%%%%%%%%%%%%%%%%%%%%%%%%%%%%%%%%%%%%%%%%%%%%%%%%%%%%%%%%%%%%
%%%%%%%%%%%%%%%%%%%%%%%%%%%%%%%%%%%%%%%%%%%%%%%%%%%%%%%%%%%%%%%%

We now look at Poisson submersions $\phi:M\to N$. We will assume that $\phi$ is
also surjective. This assumption is not too restrictive since the
image of a submersive Poisson map is the union of open subsets of symplectic leaves.

%%%%%%%%%%%%%%%%%%%%%%%%%%%%%%%%%%%%%%%%%%%%%%%%%%%%%%%%%%%%%%%%
\subsection{Poisson submersions}                               %
\label{subsec:Poisson:submer}                                  %
%%%%%%%%%%%%%%%%%%%%%%%%%%%%%%%%%%%%%%%%%%%%%%%%%%%%%%%%%%%%%%%%

Let $\phi:M\to N$ be a surjective Poisson submersion. We will denote by $\F_\phi$ the foliation of $M$ by fibers of $\phi$.
The bundle $\phi^*T^*N$ is naturally isomorphic to the conormal bundle $\nu^*(\F_\phi)$, which therefore has a natural Lie
algebroid structure. The space of 1-forms for this algebroid is the space of sections of the
normal bundle: $\X^1(\phi)=\Gamma(\nu(\F_\phi))$.

\begin{lem}
A section $[X]=\Gamma(\nu(\F_\phi))$ is closed if and only if the vector field $X$ representing it
satisfies:
\begin{equation}
\label{eq:basic:Poisson}
X(\{f\circ\phi,g\circ\phi\})=\{X(f\circ\phi),g\circ\phi\}+\{f\circ\phi,X(g\circ\phi)\},
\end{equation}
for every $f,g\in C^\infty(N)$. Such a section is exact if and only if it is represented by a hamiltonian vector field $X$.
\end{lem}

\begin{proof}
Remark first that sections of $\nu^*(\F_\phi)$ are generated by the 1-forms of the form $\phi^*\d f$, where $f\in C^\infty(N)$.
Given a function $h\in C^\infty(M)$, we compute its Lie algebroid differential relative to $\nu^*(\F_\phi)$ as:
\begin{align*}
\langle\d h,\phi^*\d f\rangle&=\rho(\phi^*\d f)(h)\\
&=\{f\circ\phi,h\}\\
&=-X_h(f\circ\phi)=\langle X_{-h},\phi^*\d f\rangle.
\end{align*}
So we conclude that $\d h=[X_{-h}]$.

On the other hand, given a 1-form $[X]\in \X^1(\phi)=\Gamma(\nu(\F_\phi))$, we compute its differential as follows:
\begin{align*}
\d[X](\phi^*\d f,\phi^*\d g)&=\rho(\phi^*\d f)(X(\phi^*\d g))-\rho(\phi^*\d g)(X(\phi^*\d f))-X(\phi^*\d f,\phi^*\d g)\\
&=\{f\circ\phi,X(g\circ\phi)\}-\{g\circ\phi,X(f\circ\phi)\}-X(\{f\circ\phi,g\circ\phi\})
\end{align*}
so the lemma follows.
\end{proof}

Since functions of the form $f\circ\phi$ are usually called basic
functions, we shall  say that   a vector field that satisfies
condition \eqref{eq:basic:Poisson}  is \textbf{Poisson along basic
  functions}. Noting that a vector field on $M$ is
$\phi$-projectable on $N$ if and only if it takes basic functions to
basic functions, we see that a projectable vector field is Poisson
along basic functions if and only  its projection  is a Poisson vector field on $N$.

\begin{cor}
For a surjective Poisson submersion $\phi:M\to N$, the vanishing of the modular class of $\phi$ is equivalent to the
existence of volume forms $\mu$ on $M$ and $\nu$ on $N$ such that the projection of  the modular vector field $X_\mu$  is  the modular vector field $X_\nu$.
\end{cor}

This corollary follows immediately from the definition of the modular class. One can also describe this in a more formal fashion, which will be useful later, as follows. Denote by $\X_\proj^k(M)\subset\X^k(M)$ the space of \textbf{projectable $k$-multivector fields} along $\phi:M\to N$. If $\Ver:=T\F_\phi=\ker \d\phi$ denotes the vertical subbundle, a vector field $X\in \X^k(M)$ is projectable if and only if
\[ \Lie_{Y}X=0,\qquad \forall Y\in\Gamma(\Ver). \]
More generally, a multivector field $A\in \X^k(M)$ is projectable if and only if
\[ \d \phi\cdot(\Lie_{Y}A)=0,\qquad \forall Y\in\Gamma(\Ver).\]

\begin{lem}
For a surjective Poisson submersion $\phi:M\to N$ the projectable multivector fields $\X_\proj^\bullet(M)$ is a subcomplex of $(\X(M),\d_\pi)$.
\end{lem}

\begin{proof}
Since $\pi:=\pi_M$ is itself projectable, for any $A\in \X_\proj^\bullet(M)$ and $Y\in\Gamma(\Ver)$, we compute:
\begin{align*}
\d\phi\cdot(\Lie_{Y}(\d_\pi A))
&= \d\phi\cdot(\Lie_{Y}[\pi,A])\\
&= \d\phi\cdot([\Lie_{Y}\pi,A]+[\pi,\Lie_{Y}A])\\
&=[\d\phi\cdot(\Lie_{Y}\pi),\d\phi\cdot A]+[\d\phi\cdot\pi,\d\phi\cdot(\Lie_{Y}A)]=0,
\end{align*}
This proves that $\d_\pi(\X_\proj^\bullet(M))\subset\X_\proj^\bullet(M)$.
\end{proof}

The cohomology of the complex $(\X^\bullet_{\proj}(M),\d_\pi)$ is denoted by $H^\bullet_{\proj,\pi}(M)$ and we refer to it as the \textbf{projectable Poisson cohomology}. Now we have a short exact sequence of complexes:
\begin{equation}
\label{short:seq:submersion}
0\rmap \X_\proj^\bullet(M) \rmap \X^\bullet(M)\oplus\X^\bullet(N)\rmap \X^\bullet(\phi)\rmap 0
\end{equation}
where the first map is given by $X\mapsto (X,\phi_*(X))$ and the the second map is given by $(X,Y)\mapsto \d\phi \cdot X-Y\circ\phi$. At the level of cohomology we obtain a long exact sequence:
\begin{equation}
\label{short:seq:submersion:cohom}
\cdots \rmap H^1_{\proj,\pi} \rmap H^1_\pi(M)\oplus H^1_\pi(N)\rmap H^1(\phi)\rmap H^2_{\proj,\pi}\rmap \cdots
\end{equation}
The map in this sequence takes $(\mod(M),\mod(N))$ to $\mod(\phi)$, and this is already true at the level of representatives. Hence, whenever $\mod(\phi)=0$, we can choose volume forms $\mu$ on $M$ and $\nu$ on $N$ such the vector field $X_{\mu,\nu}$ representing $\mod(\phi)$ vanishes, and this means that $X_\mu$ projects to $X_\nu$.

\begin{ex}
\label{ex:2-dim}
As a very simple example, consider the linear Poisson bracket on $M=\Rr^2$ which is dual to the non-abelian two-dimensional Lie algebra:
\[ \{x,y\}=x.\]
The modular vector field relative to the volume form $\mu=\d x\wedge \d y$ is $X_\mu=\frac{\partial}{\partial y}$, and so $\mod(M)\not=0$.

Now consider the two projections $\phi_i:\Rr^2\to\Rr$ on the first and the second factor. These are both Poisson submersions, but their modular classes are quite different. For the first projection, the modular vector field projects onto the zero vector field on $N=\Rr$, and we have $\mod(\phi_1)=0$. However, for the second projection the modular vector field projects onto a non-zero vector field, and indeed $\mod(\phi_2)\not=0$ (note that $\mod(N)=\mod(\Rr)=0$ but $H^1_\pi(N)\simeq C^\infty(\Rr)$).
\end{ex}

The characterization of the modular class of a map as a characteristic class, also becomes simple in the case of a Poisson submersion: the bundle $\Ver$ carries a natural representation of the Lie algebroid $\nu^*(\F_\phi)=\phi^*T^*N$, defined by:
\[ \nabla_{\phi^*\al} X:=\Lie^{T^*M}_{\phi^*\al}X, \quad (\al\in\Omega^1(N), X\in\Gamma(\Ver)).\]
Here $\Lie^{T^*M}$ denotes the Lie derivative on the Lie algebroid $T^*M$. In this way, $\wedge^\top \Ver$, as well as $\wedge^\top \Ver^*$,
become one-dimensional representations of $\phi^*T^*N$ and we have:

\begin{prop}
\label{prop:mod:subm:char:vert*}
The modular class of a surjective Poisson submersion $\phi:M\to N$ is given by
\[\mod \phi=\ch(\wedge^\top \Ver^*).\]
\end{prop}

\begin{proof}
The short exact sequence
\[0\rmap \Ver\hookrightarrow TM\rmap \phi^*TN\rmap 0\]
yields the canonical line bundle isomorphism
\[\wedge^\top T^*M\otimes \wedge ^\top \phi^*TN\cong \wedge^\top \Ver^*.\]
With this isomorphism  we can easily identify the representation
$L^\phi$ given in Proposition \ref{prop:mod:poisson:rep} with $\otimes^2\wedge^\top \Ver^*$ defined above.
So, $\mod(\phi)=\frac{1}{2}\ch(\otimes^2\wedge^\top \Ver^*)=\ch(\wedge^\top\Ver^*)$.
\end{proof}

%%%%%%%%%%%%%%%%%%%%%%%%%%%%%%%%%%%%%%%%%%%%%%%%%%%%%%%%%%%%%%%%%%%%
\subsection{Reduction: modular class of Poisson quotients}         %
\label{subsec:Poisson:quotients}                                   %
%%%%%%%%%%%%%%%%%%%%%%%%%%%%%%%%%%%%%%%%%%%%%%%%%%%%%%%%%%%%%%%%%%%%

Let $(M,\pi_M)$ be a Poisson manifold, let $(G,\pi_G)$ be a Poisson-Lie group and
let $G\times M\to M$ be a Poisson action. For simplicity, we will always assume that $G$ is connected and we will say that $(M,\pi)$ is a Poisson $G$-space.
For a $G$-space the set of $G$-invariant functions, denoted $C^\infty(M)^G$, is a Poisson
subalgebra of $C^\infty(M)$. Hence, if the action is proper and free, there is a quotient Poisson structure on
$M/G$ such that the projection $\phi:M\to M/G$ is a surjective, Poisson submersion. In this situation,
it is natural to ask what is the relationship between $\mod(M/G)$, $\mod(M)$ and the group $G$.

The following fact will be crucial to understand this relationship:

\begin{thm}
\label{thm:mod:class:quotient}
Let $G$ be a Poisson-Lie group, let $G\times M\to M$ be a proper and free Poisson action and let $\phi:M\to M/G$ be the
quotient map. The modular class of $\phi$ vanishes:
\[ \mod(\phi)=0.\]
\end{thm}

\begin{proof}
We start by choosing a $G$-invariant fiberwise volume form $\mu_G\in\Omega^\top(\F_\phi)$, a volume form $\nu\in\Omega^\top(M/G)$, and  $\{\xi^1,\dots,\xi^d\}$ any basis of $\gg$, set:
\begin{equation}
\label{eq:vol:form}
\mu=\frac{\mu_G\wedge\phi^*\nu}{\langle\mu_G,\xi^1_M\wedge\cdots\wedge\xi^d_M\rangle},
\end{equation}
where $\xi_M$ denotes the infinitesimal generator associated with $\xi\in\gg$. We claim that the modular vector field relative to $\mu$ and $\nu$ vanishes: $X_{\mu,\nu}=0$.

Using the representation $\wedge^\top\Ver^*$ given by Proposition \ref{prop:mod:subm:char:vert*}, we find that the modular vector field $X_{\mu,\nu}$ satisfies for any $f\in C^\infty(M/G)$:
\begin{align}
\label{eq:mod:class:quotient:1}
X_{\mu,\nu}(f\circ\phi) \frac{\mu_G\wedge\phi^*\nu}{\langle\mu_G,\xi^1_M\wedge\cdots\wedge\xi^d_M\rangle}&=(\Lie_{X_{f\circ \phi}}\frac{\mu_G}{\langle\mu_G,\xi^1_M\wedge\cdots\wedge\xi^d_M\rangle})\wedge \phi^*\nu\notag\\
=\frac{(\Lie_{X_{f\circ \phi}}\mu_G)\wedge\phi^*\nu}{\langle\mu_G,\xi^1_M\wedge\cdots\wedge\xi^d_M\rangle}&-\frac{\Lie_{X_{f\circ \phi}}\langle\mu_G,\xi^1_M\wedge\cdots\wedge\xi^d_M\rangle}{(\langle\mu_G,\xi^1_M\wedge\cdots\wedge\xi^d_M\rangle)^2}\mu_G\wedge\phi^*\nu.
\end{align}

On the other hand, we compute:
\begin{align*}
X_{\log|\langle\mu_G,\xi^1_M\wedge\cdots\wedge\xi^d_M\rangle|}(f\circ\phi)
	&=-X_{f\circ\phi}(\log|\langle\mu_G,\xi^1_M\wedge\cdots\wedge\xi^d_M\rangle|)\\
	&=-\frac{1}{\langle\mu_G,\xi^1_M\wedge\cdots\wedge\xi^d_M\rangle}\Lie_{X_{f\circ\phi}}\langle\mu_G,\xi^1_M\wedge\cdots\wedge\xi^d_M\rangle\\
	&=-\frac{1}{\langle\mu_G,\xi^1_M\wedge\cdots\wedge\xi^d_M\rangle}\langle\Lie_{X_{f\circ\phi}}\mu_G,\xi^1_M\wedge\cdots\wedge\xi^d_M\rangle.
\end{align*}
In other words,
\[
\langle\Lie_{X_{f\circ\phi}}\mu_G,\xi^1_M\wedge\cdots\wedge\xi^d_M\rangle
=-X_{\log|\langle\mu_G,\xi^1_M\wedge\cdots\wedge\xi^d_M\rangle|}(f\circ\phi)\langle \mu_G,\xi^1_M\wedge\cdots\wedge\xi^d_M\rangle.
\]
This shows that:
\begin{equation}
\label{eq:mod:class:quotient:2}
\frac{(\Lie_{X_{f\circ \phi}}\mu_G)\wedge\phi^*\nu}{\langle\mu_G,\xi^1_M\wedge\cdots\wedge\xi^d_M\rangle}=-X_{\log|\langle\mu_G,\xi^1_M\wedge\cdots\wedge\xi^d_M\rangle|}(f\circ\phi)\frac{\mu_G\wedge\phi^*\nu}{\langle\mu_G,\xi^1_M\wedge\cdots\wedge\xi^d_M\rangle}
\end{equation}
and also:
\begin{align}
\label{eq:mod:class:quotient:3}
\frac{\Lie_{X_{f\circ \phi}}\langle\mu_G,\xi^1_M\wedge\cdots\wedge\xi^d_M\rangle}{(\langle\mu_G,\xi^1_M\wedge\cdots\wedge\xi^d_M\rangle)^2}\mu_G\wedge\phi^*\nu
&=\frac{\langle\Lie_{X_{f\circ \phi}}\mu_G,\xi^1_M\wedge\cdots\wedge\xi^d_M\rangle}{(\langle\mu_G,\xi^1_M\wedge\cdots\wedge\xi^d_M\rangle)^2}\mu_G\wedge\phi^*\nu
\notag \\
&=-\frac{X_{\log|\langle\mu_G,\xi^1_M\wedge\cdots\wedge\xi^d_M\rangle|}(f\circ\phi)}{\langle\mu_G,\xi^1_M\wedge\cdots\wedge\xi^d_M\rangle}\mu_G\wedge\phi^*\nu.
\end{align}
Equations \eqref{eq:mod:class:quotient:1}, \eqref{eq:mod:class:quotient:2} and \eqref{eq:mod:class:quotient:3} together show that:
\[ \langle X_{\mu,\nu},\d(f\circ\phi)\rangle=X_{\mu,\nu}(f\circ\phi)=0\]
Since the forms $\d (f\circ\phi)$ generate $\Gamma(\nu^*(\F_\phi))$, we conclude that $X_{\mu,\nu}=0$ and the theorem follows.
\end{proof}

The theorem says that one can always choose volume forms $\mu\in\Omega^\top(M)$ and $\nu\in\Omega^\top(M/G)$ such that $X_{\mu}$ projects to $X_\nu$. Note that if $X_{\mu}=X_h$ is hamiltonian this does not imply that $X_\nu$ is hamiltonian, since one may not be able to choose the hamiltonian $h$ to be a $G$-invariant function. In other words, one can have $\mod(M)=0$ while $\mod(M/G)\not=0$ as illustrated by the following simple example.

\begin{ex}
\label{ex:T*G}
For any Lie group $G$, its cotangent bundle $T^*G$ is a symplectic manifold. The action of $G$ on itself by left translations
lifts to a proper and free action of $G$ on $T^*G$ by Poisson diffeomorphisms. The quotient is:
\[ \phi:T^*G\to T^*G/G\simeq \gg^*,\]
where $\gg^*$ is equipped with its canonical linear Poisson
structure. Since $T^*G$ is symplectic, one always has
$\mod(T^*G)=0$. On the other hand, $\mod(\gg^*)=0$ if and only if the Lie algebra $\gg$ is unimodular.
\end{ex}

The precise relationship between $\mod(M)$ and $\mod(M/G)$ becomes clear if we use the short exact sequence \eqref{short:seq:submersion:cohom}. Since the map $\phi:M\to M/G$ has vanishing modular class, this sequence shows that there exists a class in the projectable cohomology $H^1_{\proj,\pi}$ which maps to the pair $(\mod(M),\mod(M/G))$. In general, the map $H^1_{\proj,\pi}\to H^1_\pi(M)\oplus H^1_\pi(N)$ is not injective, but the following corollary of Theorem \ref{thm:mod:class:quotient} shows that there is a canonical choice of such a cohomology class:

\begin{cor}
\label{cor:mod:action}
Let $G$ be a Poisson-Lie group, let $G\times M\to M$ be a proper and free Poisson action and let $\phi:M\to M/G$ be the
quotient map. If $\mu_G\in\Omega^\top(F_\phi)$ is a $G$-invariant fiberwise volume form, $\nu\in\Omega^\top(M/G)$ and $\{\xi^1,\dots,\xi^d\}$ a basis of $\gg$, set
\[
\mu=\frac{\mu_G\wedge\phi^*\nu}{\langle\mu_G,\xi^1_M\wedge\cdots\wedge\xi^d_M\rangle}.
\]
Then $[X_\mu]\in H^1_{\proj,\pi}(M)$ is mapped to the pair $(\mod(M),\mod(M/G))$.
\end{cor}

Note that the class $[X_\mu]\in H^1_{\proj,\pi}(M)$ does not depend on the choices of $G$-invariant fiberwise volume form, $\nu\in\Omega^\top(M/G)$ and basis $\{\xi^1,\dots,\xi^d\}$.
Hence we can introduce:

\begin{defn}
The \textbf{modular class of the Poisson action} $G\times M\to M$ is the projectable Poisson cohomology class
\[ \mod(M,G):=[X_\mu]\in H^1_{\proj,\pi}(M). \]
\end{defn}

Since $\mod(M,G)$ is mapped to $\mod(M)$ and to $\mod(M/G)$, we have:
\begin{enumerate}[(i)]
\item if $\mod(M,G)=0$ then $\mod(M/G)=0$.
\item we can have $\mod(M)=0$, while $\mod(M,G)\not=0$ (see Example \ref{ex:T*G}).
\end{enumerate}

When $\gg$ is an unimodular Lie algebra the function $\langle\mu_G,\xi^1_M\wedge\cdots\wedge\xi^d_M\rangle$ is  constant, so the projectable modular class is represented by the modular vector field $X_\mu$, for any $G$-invariant volume form $\mu$. However, even when $\mod(M)=0$ and $\gg$ is unimodular one can still have $\mod(M,G)\not=0$ since there may not exist a $G$-invariant volume form which is also invariant under all Hamiltonian flows. In other words, $X_\mu$ is hamiltonian but one may not be able to choose the hamiltonian to be $G$-invariant as  illustrated by the following example.

\begin{ex}
On $M=\Rr^3$, with coordinates $(x,y,z)$, consider the action of $G=\Rr$ by translations in the $z$-direction. This action preserves the Poisson tensor:
\[ \pi_M=x\frac{\partial}{\partial x}\wedge \frac{\partial}{\partial y}+\frac{\partial}{\partial y}\wedge\frac{\partial}{\partial z}.\]
This Poisson structure is unimodular with invariant volume form $e^z \d x\wedge \d y \wedge \d z$.

The quotient $M/G$ can be identified with $\Rr^2$, with coordinates $(a,b)$, and Poisson structure
\[ \pi_M=a\frac{\partial}{\partial a}\wedge \frac{\partial}{\partial b}.\]
This Poisson structure is not unimodular (see Example \ref{ex:basic}), hence we must have $\mod(M,G)\not=0$.

We can compute $\mod(M,G)$ directly from its definition: we take $\mu_G=\d z\in\Omega^\top(F_\phi)$,  $\nu=\d a\wedge\d b\in\Omega^\top(M/G)$  and $\xi_M=\frac{\partial}{\partial z}$. Then,
\[
\mu=\frac{\mu_G\wedge\phi^*\nu}{\langle \mu_G,\xi_M\rangle}=\d x\wedge \d y\wedge \d z
\]
 which  is the canonical volume form in $\Rr^3$, so we obtain:
\[ \mod(M,G)=[X_\mu]=[-\frac{\partial}{\partial y}]. \]
The vector field $X_\mu=-\frac{\partial}{\partial y}$ projects to the modular vector field $X_\nu=-\frac{\partial}{\partial b}$, so that the class $\mod(M,G)$ maps to the class $\mod(M/G)$. On the other hand, $X_\mu$ is a hamiltonian vector field (e.g., one can choose as hamiltonian $h(x,y,z)=z$) and so  $\mod(M,G)$ maps to $\mod(M)=0$. But, because we cannot choose $h$ to be a $G$-invariant function, we have $\mod(M,G)\not=0$.
\end{ex}

We shall see later (Corollary \ref {cor:unimodular}) that for a proper and free \emph{hamiltonian action} one has that $\mod(M)=0$ and $\gg$ unimodular imply $\mod(M/G)=0$ (but still one can have $\mod(M,G)\not=0$!). On the other hand, we show now that when the group $G$ is compact and acts by Poisson diffeomorphisms then $\mod(M,G)$ vanishes.  In order to explain this, note that the \textbf{$G$-invariant multivector fields} form a subspace $\X^k(M;G)\subset \X_\proj^k(M)$ of the space of projectable multivector fields. In fact, we have that (since $G$ is connected) that $A\in \X^k(M)$ is $G$-invariant if and only if
\[ \Lie_{\xi_M}A=0,\quad  (\xi\in\gg),\]
and this implies that $A$ is projectable. Moreover, this also shows that $\X^\bullet(M;G)$ is a subcomplex of $(\X^\bullet(M),\d_\pi)$, so we have the inclusion of complexes:
\[ \X^\bullet(M;G)\subset  \X_\proj^k(M) \subset \X^\bullet(M). \]
Observe that, in degree zero, we have $\X^0(M;G)=\X^0_{\proj}(M)=C^\infty(M)^G$, so the natural map
\[ H^\bullet_\pi(M;G)\to H^\bullet_{\proj,\pi}(M) \]
is injective in degree 1.

\begin{rem}
\label{rem:G:invariance}
In general, the projectable modular class $\mod(M,G)$ cannot be represented by a $G$-invariant Poisson vector field, i.e., it does not lie in the subspace $H^1_\pi(M;G)\subset H^1_{\proj,\pi}(M)$. In fact, one can show that if $\chi_0\in\gg$ and $\vartheta_0\in\gg^*$ denote, respectively, the adjoint characters of $\gg^*$ and $\gg$ (see Example \ref{ex:Poisson-Lie}), then
\begin{equation}
\label{eq:invariance:modular}
\Lie_{\xi_M}X_\mu=\frac{1}{2}\left((\ad_{\gg}(\chi_0)\cdot\xi)_M +(\ad^*_{\gg^*}(\vartheta_0)\cdot\xi)_M\right),\quad (\xi\in\gg).
\end{equation}
\end{rem}

Using an averaging argument, Ginzburg in (\cite[Corollary 4.13]{G1}) proves that the morphism $H^1_\pi(M;G)\to H^1_{\pi}(M)$ is injective when the
action is by Poisson diffeomorphisms and $G$ is compact . He also conjectured that this should be true for any Poisson action of a compact Lie  group, but this is still an open conjecture. Ginzburg's result combined with the previous remark gives:

\begin{cor}
\label{cor:compact}
Let $G\times M\to M$ be a proper and free action of a compact Lie group by Poisson diffeomorphisms. If $\mod(M)=0$, then $\mod(M,G)=0$ and hence $\mod(M/G)=0$.
\end{cor}

\begin{proof}
Since the action is by Poisson diffeomorphisms, by Remark \ref{rem:G:invariance}, the modular vector field $X_\mu$ representing $\mod(M,G)$ is $G$-invariant. Therefore, the modular class of the action lies in the $G$-invariant Poisson cohomology: $\mod(M,G)\in H^1_\pi(M;G)$. Since $H^1_\pi(M;G)\to H^1_\pi(M)$ is injective, if $\mod(M)=0$ then we must have that $\mod(M,G)=0$ and the corollary follows.
\end{proof}

%When $\mod(M)=0$ and $\gg$ is not unimodular, there may exist a $G$-invariant volume form $\mu$ which is invariant under all Hamiltonian diffeomorphisms but still $\mod(M/G)\not=0$. In fact, under these assumptions, we will find
%\[
%\mod(M,G)=[X_{\log|\langle\mu_G,\xi^1_M\wedge\cdots\wedge\xi^d_M\rangle|}]
%\]
%and this class maybe non-zero if $\gg$ is not unimodular.
%Now, there are two possibilities. The class $\mod(M,G)$ projects to a non-zero class, yielding $\mod(M/G)\not=0$. This is precisely what happens in Example \ref{ex:T*G}. There is also the possibility that $\mod(M,G)\not=0$ but the projection is zero, as illustrated by Example \ref{ex:2-dim} (the first projection corresponds to the action of $\Rr$ by translations along the $y$-direction). So in this case, one has $\mod(M/G)=0$ but there is no $G$-invariant volume form on $M$ which is also invariant under all Hamiltonian diffeomorphisms.

As a curious application of these results, we close this section
with  an example of a Poisson manifold which cannot be written as a quotient of a symplectic manifold by an action of a compact Lie group.

\begin{ex}
On $M=\Rr^3-\{0\}$, with coordinates $(x,y,z)$  consider the Poisson
 manifold:
\[ \pi=(x^2+y^2+z^2)^{\frac{1}{2}}\frac{\partial}{\partial x}\wedge\frac{\partial}{\partial y}. \]
This Poisson structure is the quotient of $\Cc^2-\{0\}$, with the standard symplectic structure, by a free $\Ss^1$-action by symplectomorphisms
(see \cite[Example 4.8]{FOR}). Hence, $\mod(M,\pi)=0$. On the other hand, if we consider on $M=\Rr^3-\{0\}$ the Poisson
structure:
\[ \bar{\pi}=\left[(x^2+y^2+z^2)^{\frac{1}{2}}-1\right]\frac{\partial}{\partial x}\wedge\frac{\partial}{\partial y}, \]
we claim that $\mod(M,\bar{\pi})\not=0$. To see this, we let $\mu=\d x\wedge\d y\wedge \d z$ and compute the associated modular vector field:
\[ X_\mu=\frac{1}{(x^2+y^2+z^2)^{\frac{1}{2}}}\left(y\frac{\partial}{\partial x}-x\frac{\partial}{\partial y}\right).\]
Since $\bar{\pi}$ vanishes on the sphere of radius 1 while $X_\mu$ does not vanish there, we conclude that $X_\mu$ is not a hamiltonian vector field and
our claim follows. Since $\mod(M,\bar{\pi})\not=0$, by the corollary, this Poisson manifold is not the quotient of a symplectic manifold by
an action of a compact Lie group.
\end{ex}

%%%%%%%%%%%%%%%%%%%%%%%%%%%%%%%%%%%%%%%%%%%%%%%%%%%%%%%%%%%%%%%%
\subsection{Reduction: modular class of hamiltonian quotients} %
\label{subsec:hamilton:quotients}                              %
%%%%%%%%%%%%%%%%%%%%%%%%%%%%%%%%%%%%%%%%%%%%%%%%%%%%%%%%%%%%%%%%

We will consider now a Poisson action of $G\times M\to M$ which is hamiltonian in the sense of Lu \cite{Lu,Lu1}. This means
that there exists a smooth map $\kappa:M\to G^*$ with values in the dual, 1-connected, Poisson-Lie group, equivariant
relative to the left dressing action of $G$ on $G^*$, such that the following moment map condition holds:
\begin{equation}
\label{eq:moment:map}
\xi_M=\pi^\sharp(\kappa^*\xi^L)
\end{equation}
(here $\xi^L\in\Omega^1(G^*)$ denotes the left-invariant 1-form whose value at the identity is $\xi\in T^*_eG^*=\gg$).

%\comment{ Also check that we use moment or momentum consistently.}

The equivariance condition on the moment map implies that $\kappa:M\to G^*$ is a Poisson map, when we equip $G^*$ with its Poisson-Lie structure.
Our first result is the following.

\begin{prop}\label{prop:mod:class:moment:map}
For a proper hamiltonian action $G\times M\to M$ of a Poisson-Lie group $G$, the modular class of the moment map $\kappa:M\to G^*$  is given by
\[ \mod(\kappa)=-\brr{\vartheta_0^R\circ \kappa}, \]
where $\vartheta_0$ is the adjoint character of  $\gg$.
\end{prop}

\begin{proof}
Since the action is proper, we can choose a $G$-invariant volume form $\mu\in\Omega^\top(M)$,
and we compute:
\begin{equation}
\label{eq:fundvf:divfree}
\langle \xi^L, \d \kappa \cdot X_\mu\rangle\mu=\Lie_{\xi_M}\mu +  \eval{\kappa^*\d\xi^L,\pi}\mu=\kappa^*\eval{\d\xi^L,\pi_{G^*}}\mu,\quad (\xi\in\gg).
\end{equation}
By a result of Evens, Lu and Weinstein (see \cite{ELW}), if $\al^L\in\Omega^1(G)$ is the left invariant form associated with $\al\in\gg^*$, we have:
\begin{equation}\label{eq:eval:Poisson:Lie:group}
\eval{\d\al^L,\pi_G}= \eval{\al^L, \chi_0^L-X_{\nu^L}},
\end{equation}
where $\nu^L$ is a left-invariant volume form in $G$. Dualizing to $G^*$, we conclude that:
\begin{equation}
\label{eq:mod:ham:action}
\d \kappa \cdot X_\mu=(-X_{\nu^L}+\vartheta_0^L)\smalcirc \kappa.
\end{equation}
Using formula \eqref{eq:mod:class:Poisso:Lie:grp} for the modular class of a Poisson-Lie group, it follows that:
\[
[\d \kappa \cdot X_\mu-X_{\nu^L}\circ \kappa]=[(-2X_{\nu^L}+\vartheta_0^L)\circ \kappa]=-[\vartheta_0^R\circ \kappa]
\]
and consequently $\mod(\kappa)=-\brr{\vartheta_0^R\circ \kappa}$.
\end{proof}

\begin{cor}
\label{cor:unimodular}
Let $G\times M\to M$ be a proper and free hamiltonian action of a Poisson-Lie group $G$ with moment map $\kappa:M\to G^*$. If $\gg$ is unimodular then:
\begin{enumerate}[(i)]
\item the moment map has vanishing modular class: $\mod{\kappa}=0$.
\item if $\mod(M)=0$, then $\mod(M/G)=0$.
\end{enumerate}
\end{cor}

\begin{proof}
Part (i) follows immediately from the proposition. To prove part (ii) we start by observing that since $\gg$ is unimodular, if $\mu$ is a $G$-invariant form, equation  \eqref{eq:mod:class:Poisso:Lie:grp} shows that:
\[ \d \kappa\cdot X_\mu=-X_{\nu^L}=-\frac{1}{2}\pi_{G^*}^\sharp (\chi_0)^L=-\frac{1}{2}\d \kappa \cdot(\chi_0)_M \]
By \eqref{eq:invariance:modular}, the vector field $X_\mu+\frac{1}{2}(\chi_0)_M$ is $G$-invariant. On the other hand, the 1-form $(\chi_0)^L\in\Omega^1(G^*)$ is exact (see Example \ref{ex:Poisson-Lie}), so $(\chi_0)_M$ is a hamiltonian vector field. Therefore, the vector field  $X_\mu+\frac{1}{2}(\chi_0)_M$ has the following properties:
\begin{enumerate}[(a)]
\item it is $G$-invariant  and lies in the kernel of the moment map;
\item it is a Poisson vector field (it is a sum of Poisson vector fields);
\item it represents a class in $H^1_{\proj,\pi}(M)$ which projects to $\mod(M/G)$ (since $X_\mu$ projects to a representative of this modular class and $(\chi_0)_M$ projects to zero).
\end{enumerate}
When $\mod(M)=0$ we conclude that $X_\mu+\frac{1}{2}(\chi_0)_M$ is a hamiltonian vector field, which by (a) has a $G$-invariant hamiltonian. This shows that $\mod(M/G)=0$.
\end{proof}

\begin{rem}
Notice that in case (ii)
it may still happen that $\mod(M,G)\not=0$. The proof only shows that the combination $X_\mu+\frac{1}{2}(\chi_0)_M$ admits a $G$-invariant hamiltonian and this may not be the case for $X_\mu$.
\end{rem}

Let $G\times M\to M$ be a hamiltonian action for which the identity $e\in G^*$ is a regular value and the action
of $G$ on $\kappa^{-1}(e)$ is proper and free. The Dirac structure $L_\pi$ associated to $\pi$ pulls back to a Dirac structure on
the level set $\kappa^{-1}(e)$, which factors through the quotient map $\kappa^{-1}(e)\to \kappa^{-1}(e)/G$, yielding a Dirac structure which is
associated with a Poisson structure   $\pi_{\red}$. We shall refer to
$$M//G:=\kappa^{-1}(e)/G$$
with its natural reduced Poisson structure $\pi_{\red}$ as the \textbf{hamiltonian quotient} of $M$ by $G$ relative to the moment map $\kappa:M\to G^*$. When the action of $G$ on the whole manifold $M$ is proper and free, we obtain a commutative diagram:
\[
\xymatrix{
     &M\ar[dr]\\
\kappa^{-1}(e)\ar[ur]\ar[dr]& &M/G\\
 &M//G\ar[ur]}
\]
where the inclusions are backward Dirac maps and the projections are forward Dirac maps (see \cite{Bur} for details on Dirac structures and Dirac maps).

We will now study the relationship between the modular class of $M$, the modular class of $M/G$ and the modular class of $M//G$. The key result is the following:

\begin{thm}
\label{thm:mod:class:ham:quotient}
Let $G\times M\to M$ be a proper and free hamiltonian action of a Poisson-Lie group $G$ with moment map $\kappa:M\to G^*$ and let $\phi:M\to M/G$ be the quotient map. For any $G$-invariant fiberwise volume form $\mu_G\in\Omega^\top(\F_\phi)$, any volume form $\tau\in\Omega^\top(M//G)$, and  $\{\xi^1,\dots,\xi^d\}$ any basis of $\gg$, set:
\begin{equation}
\label{eq:vol:form:ham}
\bar{\mu}=\frac{\mu_G\wedge\phi^*\tau\wedge \kappa^*\xi_1^L\wedge\cdots\wedge \kappa^*\xi_d^L}{\langle\mu_G,\xi^1_M\wedge\cdots\wedge\xi^d_M\rangle}.
\end{equation}
The modular vector field $X_{\bar{\mu}}\in\X(M)$ is projectable and tangent to $\kappa^{-1}(e)$. Moreover, $\left.X_{\bar{\mu}}\right|_{\kappa^{-1}(e)}$ and $X_\tau\in\X(M//G)$ are $\phi$-related:
\[ \d\phi\cdot\left.X_{\bar{\mu}}\right|_{\kappa^{-1}(e)}=X_\tau\smalcirc\phi.\]
\end{thm}

\begin{rem}
Some explanations about the statement of this theorem are in
order. First, we do not distinguish between a volume form $\tau\in\Omega^\top(M//G)$ and an extension to a form defined in $M/G$ whose pullback by $ M//G\hookrightarrow M/G$ is $\tau$. This is fine since we are only interested in the relative modular class of $M//G$ in $M/G$ (see Remark \ref{rem:relative:class}). Second, in general, the form $\kappa^*\xi_1^L\wedge\cdots\wedge \kappa^*\xi_d^L$ is not basic, so the volume form \eqref{eq:vol:form:ham} is quite different from the volume form \eqref{eq:vol:form}, and $X_{\bar{\mu}}$ does not represent $\mod(M,G)$.
\end{rem}

\begin{rem} Theorem \ref{thm:mod:class:ham:quotient} can also be explained in terms of vanishing of the modular class of a map, just like the case of Poisson quotients in the previous section: the quotient map $\kappa^{-1}(e)\to \kappa^{-1}(e)/G=M//G$ is a forward Dirac map and has vanishing modular class. A long exact sequence of the sort \eqref{short:seq:submersion:cohom} then yields Theorem \ref{thm:mod:class:ham:quotient}. However, since this would require a discussion of modular classes of Dirac manifolds, we will not pursue this issue here and refer the reader for an upcoming publication.
\end{rem}

\begin{proof}[Proof of Theorem \ref{thm:mod:class:ham:quotient}]
Set $\nu^L:=\xi_1^L\wedge\cdots\wedge \xi_d^L$. First observe that
\begin{align*}
\Lie_{\xi_M}\bar\mu &= \mu_G\wedge \phi^*\tau\wedge \Lie_{\xi_M}\frac{\kappa^*\nu^L}{\langle\mu_G,\xi^1_M\wedge\cdots\wedge\xi^d_M\rangle }\\
&= \mu_G\wedge \phi^*\tau\wedge \left(\frac{\kappa^*\Lie_{\pi_{G^*}^\sharp \xi^L}\nu^L}{\langle\mu_G,\xi^1_M\wedge\cdots\wedge\xi^d_M\rangle}-\frac{\Lie_{\xi_M}\langle\mu_G,\xi^1_M\wedge\cdots\wedge\xi^d_M\rangle}{(\langle\mu_G,\xi^1_M\wedge\cdots\wedge\xi^d_M\rangle)^2} \kappa^*\nu^L\right)\\
&=\left(\eval{X_{\nu^L}\smalcirc\kappa,\xi^L} -\kappa^*\eval{\d\xi^L, \pi_{G^*}}-\vartheta_0(\xi)\right)\bar\mu.
\end{align*}
Hence, for $x\in\kappa^{-1}(e)$, we have
$\d_x\kappa\cdot X_{\bar\mu}=X_{\nu^L}(e)-\vartheta_0=0$ and we conclude that $X_{\bar\mu}$ is tangent to $\kappa^{-1}(e)$.

To see that $\left.X_{\bar{\mu}}\right|_{\kappa^{-1}(e)}$ is $\phi$-related to $X_\tau$, we notice that, for $f\in C^\infty(M/G)$,
 \begin{align*}
 \Lie_{{\pi^\sharp(\phi^*\d f)}}\phi^*\tau&=X_\tau(f)\phi^*\tau,\\
 \Lie_{{\pi^\sharp(\phi^*\d f)}}\kappa^*\nu^L&=0, \\
\ds \left(\Lie_{\pi^\sharp(\phi^*\d f)}\mu_G\right)\wedge\phi^*\tau\wedge \kappa^*\nu^L{|_{\kappa^{-1}(e)}}&=-X_{\ln|\langle\mu_G,\xi^1_M\wedge\cdots\wedge\xi^d_M\rangle|}(f\smalcirc\phi)\bar\mu{|_{\kappa^{-1}(e)}}.
\end{align*}
Now, it follows simply from the definition of $X_{\bar\mu}$ that
\[
X_{\bar\mu}(f\smalcirc\phi) |_{\kappa^{-1}(e)}= X_{\tau}(f)|_{\kappa^{-1}(e)},
\]
and since we already know that $X_{\bar\mu}$ is tangent to $\kappa^{-1}(e)$ the result follows.
\end{proof}

\begin{prop}
The relative modular class of $ M//G\hookrightarrow M/G$ is trivial if and only if $\gg$ is unimodular.
\end{prop}

\begin{proof} Let $\bar\mu$ be the volume form defined by \eqref{eq:vol:form:ham} and $\mu=\frac{\mu_G\wedge\phi^*{\bar\tau}}{\langle\mu_G,\xi^1_M\wedge\cdots\wedge\xi^d_M\rangle}$, for some volume form $\bar\tau\in\Omega^\top(M/G)$. Consider  the function  $H_{\tau,\bar\tau} \in C^\infty(M)$  such that $\bar\mu=H_{\tau,\bar\tau}\mu.$ Since $X_\mu$ and $X_{\bar\tau}$ are $\phi$-related (see Theorem \ref{thm:mod:class:quotient}), as well as $\left.X_{\bar\mu}\right|_{\kappa^{-1}(e)}$ and $\left.X_{\bar\tau}\right|_{\kappa^{-1}(e)/G}$ (see Theorem \ref{thm:mod:class:ham:quotient}),
the relative modular vector field associated to the volume forms $\tau$ and $\bar\tau$ is given by
\[X_{\tau,\bar\tau}=\left.X_{\bar\tau}\right|_{\kappa^{-1}(e)/G}-X_\tau=\left.\d\phi\cdot X_\mu\right|_{\kappa^{-1}(e)}-\left.\d\phi\cdot X_{\bar\mu}\right|_{\kappa^{-1}(e)}=\left.\d\phi\cdot X_{\ln |H_{\tau,\bar\tau}|}\right|_{\kappa^{-1}(e)}.
\]
Also, since $X_{\bar\mu}$ is tangent to $\kappa^{-1}(e)$, we have that
\begin{equation}
\label{eq:log}
\left.\d\kappa\cdot X_{\ln |H_{\tau,\bar\tau}|}\right|_{\kappa^{-1}(e)}=-\d\kappa\cdot X_{\mu}|_{\kappa^{-1}(e)}=\vartheta_0.
\end{equation}

Now, if the relative modular class is trivial, then we can choose
$\tau$ and $\bar\tau$ such that $X_{\tau}=X_{\bar\tau}|_{M//G}$,
i.e., such that ${X_{\ln |H_{\tau,\bar\tau}|}}|_{\kappa^{-1}(e)}$
is vertical and consequently $X_{\ln |H_{\tau,\bar\tau}|}$
is  tangent to $\kappa^{-1}(e)$. It follows from \eqref{eq:log} that $\gg$ must be unimodular.

Conversely, if $\gg$ is unimodular then $X_{\ln |H_{\tau,\bar\tau}|}$ is tangent to $\kappa^{-1}(e)$ and we conclude from \eqref{eq:log} that:
\[ \Lie_{\xi_M}\ln |H_{\tau,\bar\tau}|=-\eval{\xi^L,\d\kappa \cdot X_{\ln |H_{\tau,\bar\tau}|}}=0,\]
i.e., the function $\left(\ln |H_{\tau,\bar\tau}|\right)|_{\kappa^{-1}(e)}$ is $G$-invariant. This implies that the relative modular vector field $X_{\tau,\bar\tau}$ is hamiltonian and consequently the relative modular class is trivial.
\end{proof}

%\begin{cor}\label{cor:mod:quotient:moment}
%The relative modular class of $ M//G\hookrightarrow M/G$ is represented by the projection of the hamiltonian vector field %$X_{\log|\langle\mu_G,\xi^1_M\wedge\cdots\wedge\xi^d_M\rangle|}$. In particular, this class is trivial if and only if $G$ is unimodular.
%\end{cor}

%\begin{proof}
%\comment{Proof missing}
%\end{proof}

Just as the projectable Poisson cohomology, which contained the
$G$-invariant Poisson cohomology, played a crucial role in the
question of relating $\mod(M)$ and $\mod(M/G)$, we will see that there
is a certain cohomology, containing the $G$-equivariant Poisson
cohomology, introduced by Ginzburg in \cite{G1}, which plays an
analogous role in the problem of relating $\mod(M)$ and $\mod(M//G)$.

We start by recalling the definition of the $G$-equivariant Poisson cohomology. Let $G\times M\to M$ be a hamiltonian action with moment map $\kappa:M\to G^*$. The moment map gives a Lie algebra homomorphism
$\tilde{\kappa}:\gg\to\Omega^1(M)$ by setting:
\[ \tilde{\kappa}(\xi)=\kappa^*\xi^L\qquad (\xi\in\gg). \]
This is called by Ginzburg the \emph{pre-momentum map}. It gives rise to a $\gg$-module structure on $\X^\bullet(M)$ by setting:
\[ \xi\cdot P=\Lie_{\tilde{\kappa}(\xi)}P \qquad (\xi\in\gg),\]
where we use the Lie derivative operator of the Lie algebroid $T^*M$. Then the Poisson differential $\d_\pi:\X^\bullet(M)\to\X^{\bullet+1}(M)$
and the degree $-1$ operator:
\[ i_\xi:\X^\bullet(M)\to\X^{\bullet-1}(M),\ i_\xi P:=i_{\tilde{\kappa}(\xi)} P \qquad (\xi\in\gg),\]
turns $\X^\bullet(M)$ into a $\gg$-differential complex. Since one has a $\gg$-differential complex, one can proceed as usual to define the
equivariant cohomology via either the Cartan model or the Weil model: The \textbf{equivariant Poisson cohomology} of a hamiltonian $G$-manifold $(M,\pi)$ is the cohomology of the complex $(\X^\bullet_G(M),\d_{\pi,G})$ and will be denoted by $H_{\pi,G}(M)$.

This construction of the equivariant Poisson cohomology is valid for any hamiltonian action. For proper and free actions we have the following interpretation, also due to Ginzburg:

\begin{thm}[Ginzburg \cite{G1}]
If $G\times M\to M$ is a proper and free hamiltonian action of a Poisson-Lie group $G$ with
moment map $\kappa:M\to G^*$, then the equivariant Poisson cohomology coincides with the cohomology of the subcomplex $\X^\bullet(k)^G\subset \X^\bullet(M)$ formed by the $G$-invariant multivector fields tangent to the $\kappa$-fibers.
\end{thm}

We will need to consider an intermediate complex $\X^\bullet(\kappa)^G\subset \X^\bullet(\kappa)\subset \X^\bullet(M)$ which is the complex formed by all multivector fields which are tangent to the single fiber $\kappa^{-1}(e)$ and whose restriction to this submanifold is projectable. The corresponding cohomology will be denoted $H^\bullet_{\pi,\proj}(\kappa)$. The inclusion of complexes induces a map in cohomology:
\[ H_{\pi,G}^\bullet(M) \longrightarrow H^\bullet_{\pi,\proj}(\kappa),\]
which is injective in degree one. Moreover, we have a diagram
\begin{equation}
\label{diagram:eq:cohomology}
\xymatrix{
    & H^\bullet_{\pi}(M)\\
H^\bullet_{\pi,\proj}(\kappa)\ar[ur]\ar[dr]& \\
    & H^\bullet_{\pi}(M//G)}
\end{equation}
where the upper arrow is the forgetful morphism and the lower arrow is obtained by first restricting to $\kappa^{-1}(e)$ and then projecting to $M//G$.

Let us observe now that two volume forms of the type \eqref{eq:vol:form:ham} differ by a $G$-invariant function. Hence, for any such form $\bar{\mu}$ the Poisson vector filed $X_{\bar{\mu}}$ defines a cohomology class in $H^1_{\pi,\proj}(\kappa)$ and we set:

\begin{defn}
Let $G\times M\to M$ be a proper and free hamiltonian action of a Poisson-Lie group $G$ with moment map map $\kappa:M\to G^*$. The \textbf{equivariant modular class} of the hamiltonian action is the class:
\[ \mod_G(M)=[X_{\bar{\mu}}]\in H^1_{\pi,\proj}(\kappa),\]
where $\bar{\mu}\in\Omega^\top(M)$ is any volume form of type \eqref{eq:vol:form:ham}  in a neighborhood of $\kappa^{-1}(e)$.
\end{defn}

The fundamental result about the equivariant modular class, which motivates also its definition, is the following:

\begin{thm}
Let $G\times M\to M$ be a proper and free hamiltonian action of a Poisson-Lie group $G$. In the
diagram \eqref{diagram:eq:cohomology} the equivariant modular class
$\mod_G(M)$  is mapped  to the modular classes $\mod(M)$ and $\mod(M//G)$. In particular, if $\mod_G(M)=0$ then $\mod(M//G)=0$.
\end{thm}

The proof should be clear using Theorem \ref{thm:mod:class:ham:quotient}. For a free and proper action of a compact Lie group $G$ acting by Poisson diffeomorphisms, a result of Ginzburg (see \cite{G1}) states that the forgetful morphism $H^1_{\pi,G}(M)\to H_{\pi}^1(M)$ is injective. Since in this case the equivariant modular class lies in $H^1_{\pi,G}(M)$ we deduce that:

\begin{cor}
Let $G\times M\to M$ be a free and proper hamiltonian action of a compact Lie group by Poisson diffeomorphisms. If $\mod(M)=0$ then $\mod(M//G)=0$.
\end{cor}

%%%%%%%%%%%%%%%%%%%%%%%%%%%%%%%%%%%%%%%%%%%%%
%%%%%%%%%%%%%%%%%%%%%%%%%%%%%%%%%%%%%%%%%%%%%
%%%%%%%%%%%%%%%%%%%%%%%%%%%%%%%%%%%%%%%%%%%%%
\section{Modular characters}                %
\label{sec:groupoids}                       %
%%%%%%%%%%%%%%%%%%%%%%%%%%%%%%%%%%%%%%%%%%%%%
%%%%%%%%%%%%%%%%%%%%%%%%%%%%%%%%%%%%%%%%%%%%%
%%%%%%%%%%%%%%%%%%%%%%%%%%%%%%%%%%%%%%%%%%%%%

It is well-known that Poisson manifolds are infinitesimal objects whose global counterparts are symplectic groupoids.
Also, the modular class of a Poisson manifold has a symplectic groupoid analogue: the modular character. We will see now that
virtually all results about the modular class of a Poisson map that we
have discussed  in the preceding sections have a groupoid version.

\subsection{The modular character of a Lie groupoid}

In this section, we start by recalling a few facts about modular characters.

Let $\G$ be a Lie groupoid over $M$ and let $\s,\t: \G\to M$ be its source and target maps. We will assume, for simplicity, that $\G$ is source connected (i.e., that the fibers of $\s$, and hence $\t$, are connected). Let $\epsilon:E\to M$ be a (left) representation of $\G$, i.e., a vector bundle together with a smooth, fiberwise linear action of $\G$ on $E$ with $\epsilon$ as the moment map. We shall only be interested in the case of 1-dimensional representations, i.e., line bundles $L\to M$.

Assume first that $L\to M$ is an oriented line bundle with a global non-vanishing section $\mu:M\to L$. Then, there exists a smooth function $c_\mu:\G\to \Rr^\times$, with values in the multiplicative group of real numbers, such that for each $g\in\G$ we have
$$g\cdot \mu(\s(g))={c_\mu(g)} \mu(\t(g)).$$
One checks easily that:

\begin{prop}
The 1-cochain $\log(c_\mu):\G\to \Rr$ is closed in the differential cohomology of $\G$ with trivial coefficients. Its cohomology class does not depend on the section $\mu$ chosen.
\end{prop}

The cohomology class defined by the 1-cochain $c_\mu$ is called the \textbf{characteristic class} of $L$:
\[ \ch(L)=[\log(c_\mu)]\in H^1_\diff(\G). \]
One checks easily that if $L,L_1,L_2\in\Rep(\G)$, then:
\[ \ch(L^*)=-\ch(L),\qquad \ch(L_1\otimes L_2)=\ch(L_1)+\ch(L_2).\]
Also, if $(\Phi,\phi):\G\to \H$ is a morphism of Lie groupoids, and $L\in\Rep(\H)$ then there is a pull-back representation $\phi^{*}L\in\Rep(\G)$ and we have:
\[ \ch(\phi^{*}L)=\Phi^*\ch(L),\]
where $\Phi^*:H^\bullet_\diff(\H)\to H^\bullet_\diff(\G)$ is the map induced by $\Phi$ at the level of cohomology.

Finally, if $L$ is not orientable, then one defines its characteristic class to be the square root of the characteristic class of the
representation $L\otimes L$, so the formulas above still hold for non-orientable line bundles.

If $\G\tto M$ is a Lie groupoid with Lie algebroid $A\to M$, then $\G$ has a canonical representation on the line bundle $L^A:=\wedge^{\top}A\otimes\wedge^{\top}T^*M$:
for each $g\in\G$ let $b:M\to \G$ be a (local) bisection at $g$ and $i_b:\G\to \G$ the (local) inner automorphism defined by $b$
\[
\mathrm{I}_b(h)=b(t(h))\cdot h\cdot b(s(h))^{-1}, \quad (h\in\G).
\]
% Notice that $\d_{1_{s(g)}}\mathrm{I}_b({A_{s(g)}})=\d R_{g^{-1}}\d L_b(A_{s(g)})\subset A_{t(g)}$ and $\d_{1_{s(g)}} \mathrm{I_b}(T_{s(g)}M)\subset %T_{t(g)}M$.
One may see that, for each $\nu\otimes\mu\in L^A_{s(g)}$,
$$\ds \d_{1_{s(g)}}\mathrm{I}_b \nu \otimes \mathrm{I}_{b^{-1}}^*\mu$$ does not depend on the bisection $b$ chosen and  defines a representation of $\G$ on $L^A$. Then we set:

\begin{defn}
The \textbf{modular character} of a Lie groupoid $\G\tto M$ relative to a nowhere vanishing section
$\mu\in\Gamma(L^A)$ is the characteristic cocycle $c_{\mu}:\G\to\Rr^\times$ of the representation $L^A$.
The \textbf{modular class} of $\G$ is the characteristic class:
\[ \mod(\G):=[\log(c_{\mu})]\in H^1_\diff(\G).\]
\end{defn}

If $A$ is the Lie algebroid of the Lie groupoid $\G$, every
representation $E\to M$ of $\G$ induces a representation of $A$, and
there is a van Est map $$\VE:H^\bullet_\diff(\G;E)\to H^\bullet(A;E).$$
This map is described in detail in \cite{Crainic}, where one can also
find conditions that make this map into an isomorphism. We shall only
be interested in degree one cohomology, in which case $\VE$ is always
injective (recall that $\G$ is assumed to be source
connected). Moreover, when $\G=\G(A)$ is the source 1-connected
groupoid  of $A$, the van Est map is an isomorphism in degree 1: the integration map which associates to a Lie algebroid cocycle $\al\in\Omega^1(A)$ the groupoid cocycle
\[ c([a])=\int_a \al,\]
gives the inverse of $\VE$ at the level of cohomology.

Under this correspondence, the modular class of the Lie groupoid $\G$ corresponds to the modular class of its Lie algebroid $A$ and vice-versa:

\begin{prop}
If $\G$ is a Lie groupoid with Lie algebroid $A$, then the canonical representation of $\G$ on $L^A$ induces the canonical representation of $A$ on $L^A$, and:
\[ \mod(A)=\VE(\mod(\G)). \]
Conversely, let $\G(A)$ be the source $1$-connected
Lie groupoid of the Lie algebroid $A$. If $\mu\in\Gamma(L^A)$ is a global section, the corresponding modular cocycles $\al_\mu\in\Omega^1(A)$ and $c_\mu:\G\to\Rr^\times$ correspond to each other and we have:
$$c_\mu([a])=\exp(\int_a \al_\mu).$$
\end{prop}

Given a morphism of Lie groupoids $\Phi:\G\to \H$ one defines the \textbf{modular class of the groupoid morphism} by:
\[ \mod(\Phi):=\mod(\G)-\Phi^*\mod(\H)\in H^1_\diff(\G).\]
It follows from the previous proposition that this class is related to the modular class of the induced morphism at the level of the Lie algebroids $\Phi_*:A\to B$ by the van Est map:
\[ \mod(\Phi_*)=\VE(\mod(\Phi)).\]
Moreover, when $\G=\G(A)$, if one chooses global sections of $L^A$ and $L^B$, the cocycles representing the modular classes of the morphisms are related by integration.

\subsection{The modular character of a symplectic groupoid}
Let us turn now to the global objects integrating Poisson manifolds, i.e., symplectic groupoids.

Recall that if $(\G,\Omega)$ is a symplectic groupoid, then the base $M$ carries a canonical Poisson structure for which:
\begin{enumerate}[(a)]
\item The source (respectively, target) map is a complete Poisson (resp., anti-Poisson) map;
\item The Lie algebroid  of $\G$ is canonically isomorphic to the cotangent Lie algebroid $T^*M$.
\end{enumerate}
Let us fix a volume form $\mu\in\Omega^\top(M)$. Then  the modular
vector field $X_\mu$  is a cocycle representing the modular class of the Poisson manifold $M$. On the other hand, since $\mu\otimes\mu$ defines a global section of the bundle $L^{T^*M}=T^*M\otimes T^*M$, we also obtain a function $c_\mu:\G\to\Rr^\times$ which represents the modular class $\mod(\G)$. Note that these modular classes are related by a formula which differs from the one above by a factor of $2$:
\[ 2\mod(M)=\mod(T^*M)=\VE(\mod(\G)).\]

The presence of the symplectic structure provides additional features to the construction: we can form the (local) hamiltonian flow associated with the hamiltonian function $\frac{1}{2}\log(c_\mu)$. This gives rise to a (local) 1-parameter group $\Phi_t:\G\to\G$ of symplectic groupoid automorphisms which then must cover a 1-parameter group of Poisson diffeomorphisms $\phi_t:M\to M$:
\[
\xymatrix{
\G\ar@<.5ex>[d]\ar@<-.5ex>[d]\ar[r]^{\Phi_t}&\G\ar@<.5ex>[d]\ar@<-.5ex>[d]\\
M\ar[r]_{\phi_t}& M
}
\]
One  easily checks  that:

\begin{prop}
The 1-parameter group of Poisson diffeomorphisms $\phi_t:M\to M$ coincides with the flow of the Poisson vector field $X_\mu$.
\end{prop}

Since the van Est map is injective in degree one, we have that $\mod(\G)=0$ if and only if $\mod(M)=0$. In fact, the Poisson vector field $X_\mu$ is the hamiltonian vector field associated with $h:M\to \Rr$ if and only if the cocycle $\frac{1}{2}\log(c_\mu)$ is exact:
\[  \frac{1}{2}\log(c_\mu)=h\circ\t-h\circ\s. \]
In this case, the 1-parameter group of automorphisms $\Phi_t:\G\to \G$ is inner: it is given by conjugation by the 1-parameter group of Lagrangian bisections $b_t:M\to\G$ which are obtained by moving the identity bisection by the flow of the hamiltonian vector field $X_{h\circ\s}$.

If we start with a Poisson manifold $(M,\pi_M)$, it is well known that its cotangent Lie algebroid $T^*M$ need not be integrable. However, for integrable Poisson manifolds, the Weinstein groupoid $\Sigma(M):=\G(T^*M)$ is automatically a symplectic groupoid. It follows from the general theory for Lie groupoids that if we fix a volume form $\mu\in\Omega^\top(M)$ the modular character $c_\mu$ (relative to $\mu\otimes\mu$) and the modular vector field $X_\mu$ are related by:
\[  \frac{1}{2}\log(c_\mu([a]))=\int_a X_\mu, \]
for any $A$-path $a:I\to A$.

\subsection{The modular character of a Poisson map}
Given a Poisson map $\phi:M\to N$, what is the groupoid (integrated) version of $\mod(\phi)$? Note that a general Poisson map between integrable Poisson manifolds does not induce a morphism of the corresponding symplectic groupoids. Hence, the results above, relating the modular classes of morphisms of Lie groupoids and the modular classes of Lie algebroids, do not have an analogue for Poisson maps.

Instead, as we have pointed out before, the key fact is that the graph of a Poisson map $\phi:M\to N$ is a coisotropic submanifold, $\graf(\phi)\subset M\times\overline{N}$. Recall from Section \ref{subsec:mod:Poisson:map} that this means that the conormal bundle to the graph gives a Lie subalgebroid $\nu^*(\graf(\phi))\subset T^*M\times T^*\overline{N}$, and that there is an isomorphism $\phi^{*}T^*N\simeq \nu^*(\graf(\phi))$.
\begin{equation}
\label{eq:diag:relative}
\xymatrix{
\phi^{*}T^*N\ar[d]\ar[r]^-{(i,j)}& \nu^*(\graf(\phi)\subset T^*M\times T^*\overline{N}\ar@<-6ex>[d]\ar@<7ex>[d]\\
M\ar[r] &\graf(\phi))\subset M\times \overline{N}}
\end{equation}
Lie subalgebroids of integrable algebroids are necessarily integrable. Hence, assuming that both $M$ and $N$ are integrable Poisson manifols, it follows that the Lie algebroid $\nu^*(\graf(\phi))$ is also integrable, and so is $\phi^{*}T^*N$. However, it is important to note that a Lie subalgebroid, unlike the case of Lie subalgebras, may not integrate to a Lie subgroupoid (see \cite{MM}).

Using the results of \cite{CaFe1} (see, also, \cite{Fernandes0}) we conclude that:

\begin{prop}
\label{prop:grp:morph}
Let $\phi:M\to N$ be a Poisson map between integrable Poisson
manifolds. Then $\phi^{*}T^*N$ and $\nu^*(\graf(\phi))$ are integrable
 Lie  algebroids and the diagram \eqref{eq:diag:relative} integrates to a diagram of groupoid morphisms:
\[
\newdir{ (}{{}*!/-5pt/@^{(}}
\xymatrix{
\G(\phi^{*}T^*N)\ar[r]^-{(I,J)}\ar@<.5ex>[d]\ar@<-.5ex>[d]& \Lcal_\phi\ar@<-2pt>@{ (->}[r]\ar@<.5ex>[d]\ar@<-.5ex>[d]
&\Sigma(M)\times \overline{\Sigma(N)}\ar@<.5ex>[d]\ar@<-.5ex>[d]\\
M\ar[r] &\graf(\phi)\ar@<-2pt>@{ (->}[r]& M\times \overline{N}}
\]
where the first map  of the top row  is an isomorphism and the second map is a Lagrangian immersion (possibly non-injective).
\end{prop}

Now we can finish our task: choose volume forms $\mu\in\Omega^\top(M)$ and $\nu\in\Omega^\top(N)$, so these define the modular characters $c_\mu:\Sigma(M)\to\Rr^\times$ and $c_\nu:\Sigma(N)\to\Rr^\times$. Denote by $I:\G(\phi^{*}T^*N)\to\Sigma(M)$ and $J:\G(\phi^{*}T^*N)\to\Sigma(N)$ the groupoid morphisms integrating the Lie algebroid morphisms $i:\phi^{*}T^*N\to T^*M$ and $j:\phi^{*}T^*N\to T^*N$.

\begin{defn}
The \textbf{modular character of the Poisson map $\phi:M\to N$} is the groupoid morphism:
\[
c_{\mu,\nu}:\G(\phi^{*}T^*N)\to\Rr^\times, \quad g\mapsto c_\mu(I(g))c_\nu(J(g))^{-1}.
\]
\end{defn}

It follows immediately that:

\begin{cor}
Let $\phi:M\to N$ be a Poisson map and choose $\mu\in\Omega^\top(M)$ and $\nu\in\Omega^\top(N)$. Then the modular character of $\phi$ is given by:
\[ \frac{1}{2}\log(c_{\mu,\nu}([a]))=\int_a X_{\mu,\nu}, \]
where $X_{\mu,\nu}$ is the modular vector field of $\phi$ relative to $\mu$ and $\nu$. In particular, under the van Est map:
\begin{align*}
H^1(\G(\phi^{*}T^*N))&\stackrel{\VE}{\longrightarrow} H^1_\pi(\phi)\\
\frac{1}{2}[\log(c_{\mu,\nu})]&\longmapsto \mod(\phi).
\end{align*}
\end{cor}

\subsection{Complete Poisson maps}
Recall that a Poisson map $\phi:M\to N$ is called \emph{complete} if whenever $X_h$ is a complete hamiltonian vector field on $N$ then $X_{h\circ\phi}$ is a complete vector field on $M$.
The category whose objects are integrable Poisson manifolds and the morphisms are complete Poisson maps plays an fundamental role in Poisson geometry. This is illustrated by a deep theorem of Crainic and Fernandes \cite{CF2} stating that a Poisson manifold $M$ is integrable iff it admits a complete symplectic realization, i.e., if there is a symplectic manifold $S$ and a surjective, complete, Poisson submersion $\phi:S\to M$. We will see that for complete Poisson maps the theory of modular characters can be made more explicit.

A basic fact is the following:

\begin{prop}
Let $\phi:M\to N$ be a complete Poisson map between integrable Poisson manifolds $(M,\pi_M)$ and $(N,\pi_N)$. Then there is a natural left action of the Lie groupoid $\Sigma(N)$ on $M$ with moment map $\phi$.
\end{prop}

\begin{proof}
We construct the action as follows. For each  $g=\brr{a}\in\Sigma(N)$
and  $m\in M$  such that $\phi(m)=\s(g)$, consider
$\widetilde{\gamma}$, the unique curve on $M$ starting  at  $m$ such that $\ds \phi(\widetilde\gamma(t))=\gamma(t)$ and $\ds \pi_M^\sharp\phi^*_{\widetilde\gamma(t)}a(t)=\dot\gamma(t)$.
Notice that the curve $\widetilde a(t)=\phi^*_{\widetilde\gamma(t)}a(t)$ is a cotangent path in $M$. The action of $\Sigma(N)$ on $M$ is then given by:
\[ g\cdot m:=\tilde\gamma(1)=\t(\brr{\widetilde a}). \]
This is well defined because if $a_0$ and $a_1$ are cotangent homotopic paths then, by the results of \cite{CF1}, the corresponding paths $\tilde\gamma_0$ and $\tilde\gamma_1$ are also homotopic relative to the end points, and in particular $\tilde\gamma_0(1)=\tilde\gamma_1(1)$.
\end{proof}

Given a complete Poisson map $\phi:M\to N$ we denote by $\phi^*\Sigma(N)\tto M$ the Lie groupoid defined by the action of $\Sigma(N)$ on $M$ with moment map $\phi$, i.e.
\[
 \phi^*\Sigma(N):=\set{(m,g)\in M\times\Sigma(N): \s(g)=\phi(m)},
\]
with source and target maps:
\[ \s(m,g):=m, \quad\t(m,g):=g\cdot m. \]

\begin{prop}
The Lie groupoid $\phi^*\Sigma(N)$ is a source 1-connected Lie groupoid integrating the Lie algebroid $\phi^*T^*N$, i.e., $\phi^*\Sigma(N)\simeq\G(\phi^*T^*N)$. Moreover, the two Lie groupoid morphisms $I:\G(\phi^*T^*N)\to \Sigma(M)$ and $J:\G(\phi^*T^*N)\to \Sigma(N)$ correspond to:
\begin{align*}
I:\phi^*\Sigma(N)&\to \Sigma(M) & J:\phi^*\Sigma(N)&\to \Sigma(N) \\
  (m,[a])&\mapsto [\tilde a] & (m,[a])&\mapsto [a].
\end{align*}
\end{prop}

\begin{proof}
Notice that the source fibers of $\phi^*\Sigma(N)$ are diffeomorphic to the source fibers of $\Sigma(N)$, so they are 1-connected. If we compute $I_*:\phi^*T^*M\to T^*M$ and $J_*:\phi^*T^*M\to T^*N$ we obtain the algebroid morphisms $i$ and $j$, so the result follows.
\end{proof}

We conclude that for a complete Poisson map $\phi:M\to N$, the Lie
groupoid $\phi^*\Sigma(N)\simeq \Lcal_\phi$ is a Lagrangian
subgroupoid of $\Sigma(M)\times \overline{\Sigma(N)}$, a much better
statement than the general situation given by Proposition
\ref{prop:grp:morph}, where the Lagrangian subgroupoid is only
immersed. As a corollary, we obtain a  simple expression for the modular character of a complete Poisson map:

\begin{cor}
Let $\phi:M\to N$ be a complete Poisson map between integrable Poisson manifolds. Its modular character $c_{\mu,\nu}:\phi^*\Sigma(N)\to\Rr^\times$ relative to volume forms $\mu\in\Omega^{\top}(M)$ and $\nu\in\Omega^\top(N)$ is given by:
\[ \frac{1}{2}\log(c_{\mu,\nu}(m,[a]))=\int_{\widetilde a} X_{\mu}-\int_a X_{\nu}, \]
where $\widetilde{a}$ is the unique cotangent path in $M$ starting at $m$ and lifting the cotangent path $a$.
\end{cor}

\begin{ex}[Poisson submanifolds]
Let $N$ be a Poisson submanifold of $M$. The inclusion map $\phi:N\hookrightarrow M$ is obviously a complete Poisson map. The action groupoid  $\phi^*\Sigma(N)$ is just the restriction of $\Sigma(M)$ to $N$ which we denote by $\Sigma_N(M)$. This groupoid has an obvious representation on the normal bundle $\nu(N)$ and, hence, a representation on $\wedge^\top \nu^*N$. It follows from Theorem \ref{thm:relative:mod:class} that the characteristic class of this representation corresponds under the van Est map to the relative modular class of the Poisson submanifold $N$ (up to a factor of $1/2$). This shows, for example, that for a proper Poisson manifold $M$ (i.e., a Poisson manifold $M$ whose Weinstein groupoid $\Sigma(M)$ is proper) the relative modular class of any Poisson submanifold vanishes.
\end{ex}

\begin{ex}[Surjective submersions]
Now suppose that the complete Poisson map $\phi:M\to N$ is a surjective submersion. By linearization of the natural action of $\Sigma(N)$ on $M$, we have  a representation  of $\phi^*\Sigma(N)$ on $\Ver$ with moment map $p:\Ver\to M$, the canonical projection:
\[ (m,g)\cdot v=\d_m T_g (v),\]
where $T_g:p^{-1}(\s(g))\to p^{-1}(\t(g))$ is the action by $g\in\Sigma(N)$. Hence, we also obtain a representation of $\phi^*\Sigma(N)$ on $\wedge^\top \Ver^*$, and it follows from Proposition \ref{prop:mod:subm:char:vert*} that the characteristic class of this representation corresponds under the van Est map to the modular class of $\phi$ (up to a factor of $1/2$).
\end{ex}

\subsection{Poisson quotients}

Let  $\Phi:G\times M\to M$ be a Poisson action of a complete, connected Poisson-Lie group $(G,\pi_G)$ on an integrable Poisson manifold $(M,\pi)$. Our aim is to explain the relationship between the modular characters of the symplectic groupoids integrating $M$ and $M/G$. First, let us recall how one can obtain a symplectic groupoid integrating $M/G$ as a symplectic quotient of $\Sigma(M)$.

The induced map $j:T^*M\to \gg^*$ defined by $$<j(\al),\xi>=\xi_M(\al), \quad (\al\in\Omega^1(M),\; \xi\in\gg),$$ is a Lie bialgebroid morphism and integrates to the morphism of Lie groupoids $J:\Sigma(M)\to G^*$ given by:
$$
J([a])=[j\smalcirc a], \quad ([a]\in\Sigma(M)).
$$
It was proved in \cite{FP} that there exists a lifted hamiltonian global action $G \times \Sigma(M)\to \Sigma(M)$ with  momentum map $J:\Sigma(M)\to G^*$.
Since $J$ is a Poisson map, it is $G$-equivariant
$$
J(g\,[a])=g\,J([a]),\quad (g\in G, \,[a]\in\Sigma(M))
$$
and it follows from the groupoid morphism property of $J$ that the following  twisted multiplicative property holds:
$$
g([a].[a'])=({}^{J([a]')}g \,[a]).(g \, [a']), \quad (g\in G,\, ([a],[a'])\in\Sigma(M)^{(2)}).
$$
In the above formulas ${}^{u}g$ means the left dressing action of $u\in G^*$ on  $g\in G$, while $g u$ means the left dressing action of $g\in G$ on $u\in G^*$.

%The lifted action does not preserve $\t$-fibers, however the modified $G$-action on $\Sigma(M)$ defined by:
%\begin{equation}\label{eq:modified:action:groupoid}
%g\odot [a]=(g\,[a]^{-1})^{-1},
%\end{equation}
% induces the same action on the identity section and the induced $G$-action on the Lie algebroid $\A(\Sigma(M))=T^*M$ is the cotangent lifted action.

The twisted multiplicative and Poisson properties imply that $J^{-1}(e)\in\Sigma(M)$ is a coisotropic Lie subgroupoid whenever $e$ is a regular value of $J$. This is the case if the action of $G$ on $M$ is proper and free and in this case one can show that the symplectic quotient:
\[ \Sigma(M)//G:=J^{-1}(e)/G,\]
is a symplectic groupoid integrating $M/G$. Note that, in general, one has $\Sigma(M)//G\ne \Sigma(M/G)$, since the former may fail to be source 1-connected (see \cite{FP}).

Let us now choose a volume form $\mu\in\Omega^\top(M)$ of type \eqref{eq:vol:form} and let $c_\mu:\Sigma(M)\to\Rr^\times$ be the modular character of $\Sigma(M)$

\begin{thm}
The $1$-cocycle $c_\mu:J^{-1}(e)\to \Rr$ is $G$-invariant and it induces a groupoid morphism $ \bar c_\mu:J^{-1}(e)/G\to \Rr$. Moreover, if $J^{-1}(e)/G$ is source connected, we have:
$$\VE[\log \bar c_\mu]=2\mod (M/G).$$
\end{thm}

\begin{proof}
Equation  \eqref{eq:invariance:modular} and the fact that $G$ is connected allow us to prove that
\[
\log c_{\mu}((g\,[a])[a]^{-1})+\log\mathrm{Ev}_{\chi_0}(J((g\, [a])[a]^{-1}))-\log\mathrm{Ev}_{\vartheta_0}(({}^{J([a]^{-1})}g)g^{-1})=0
%-\log c_\mu[a] - \log\mathrm{Ev}_{\chi_0}\smalcirc J([a])+\log\mathrm{Ev}_{\vartheta_0}(g)=0.
\]
Here, for $\xi$ in the center of the Lie algebra $\gg$ we have denoted by $\mathrm{Ev}_\xi:G\to\Rr$ the corresponding group homomorphism.

Now, the $G$-invariance of $\log c_{\mu}:J^{-1}(e)\to\Rr$  is an immediate consequence of this last equation. Finally,
Theorem \ref{thm:mod:class:quotient} guarantees that $$\VE\left[\frac{1}{2}\log \bar c_{\mu}\right]=\mod(M/G).$$
\end{proof}

%\subsection{Morita equivalence and generalized morphisms}

%**** GENERALIZED POISSON MORPHISMS ???? ****

%**** INVARIANCE UNDER MORITA EQUIVALENCE ****

% -----------------------------------------------------------------------

% -----------------------------------------------------------------------

\begin{thebibliography}{10}
\bibitem{Bur} H.~Bursztyn, A brief introduction to Dirac manifolds. Preprint arXiv:1112.5037

\bibitem{CaFe1}
	A.S.~Cattaneo and G.~Felder, Coisotropic submanifolds in Poisson
	geometry and branes in the Poisson sigma model,
	\emph{Lett.~Math.~Phys.~}\textbf{69} (2004), 157--175.

\bibitem{Crainic} M.~Crainic, Differentiable and algebroid cohomology, van Est isomorphisms, and characteristic classes.
	\emph{Comment.~Math.~Helv.~}\textbf{78} (2003), no.~4, 681--721.

\bibitem{CF0} M.~Crainic and R.L.~Fernandes, Lectures on Integrability of Lie Brackets, to appear in \emph{Geometry \& Topology Monographs} \textbf{17} (2010).

\bibitem{CF1} M.~Crainic and R.L.~Fernandes, Stability of symplectic leaves.
	\emph{Invent.~Math.~} \textbf{180}, no. 3, (2010), 481--533.

\bibitem{CF2} M.~Crainic and R.L.~Fernandes, Integrability of {P}oisson brackets,
	\emph{J.~Differential Geom.~}\textbf{66} (2004), no.~1, 71--137.

\bibitem{DufourZung} J.P.~Dufour and N.-T.~Zung, \emph{Poisson structures and their normal forms}.
	Progress in Mathematics, vol.~242, Birkh\"auser, Basel, 2005.
	
\bibitem{ELW} S.~Evens, J.-H.~Lu and A.~Weinstein, Transverse
  measures, the modular class and a cohomology pairing for Lie algebroids.
  \emph{Quart.~J.~Math.~Oxford (2)} \textbf{50} (1999), 417--436.

\bibitem{Fernandes0} R.L.~Fernandes, The symplectization functor, in Proceedings of XV International Workshop on Geometry and Physics
	(Puerto de la Cruz, Tenerife, 2006), \emph{Publ.~R.~Soc.~Mat.~Esp.~}\textbf{11} (2008) 67-82.
	
\bibitem{Fernandes1} R.L.~Fernandes, Lie algebroids, holonomy and characteristic classes.
  \emph{Adv.~Math.~}\textbf{170} (2002), no.~1, 119--179.

\bibitem{Fernandes2} R.L.~Fernandes, Connections in Poisson Geometry I: Holonomy and Invariants.
	\emph{J. Diff. Geom.} \textbf{54}, (2000) 303--366.
	
\bibitem{FOR} R.L.~Fernandes, J.-P.~Ortega and T.~Ratiu, The momentum map in Poisson geometry,
	\emph{Amer.~J.~of Math.~} \textbf{131}, no. 5  (2009) 1261-1310.

\bibitem{FP} R.L.~Fernandes and D.I.~Ponte, Integrability of Poisson-Lie group actions,
    \emph{Lett. Math. Phys.} \textbf{90}  (2009) 137--159.
	
\bibitem{G1} V.L.~Ginzburg, Equivariant Poisson cohomology and a spectral sequence associated with a moment map.
	\emph{Internat.~J.~Math.~}\textbf{10} (1999),  no.~8, 977--1010.

\bibitem{G2} V.L.~Ginzburg, Momentum mappings and Poisson cohomology.
	\emph{Internat.~J.~Math.~}\textbf{7} (1996),  no.~3, 329--358.

\bibitem{GG} V.L.~Ginzburg and A.~Golubev, Holonomy on Poisson manifolds and the modular class.
  \emph{Israel J.~Math.~}\textbf{122} (2001), 221--242.

\bibitem{GL} V.L.~Ginzburg and J.-H.~Lu, Poisson cohomology of Morita-equivalent Poisson manifolds.
  \emph{Internat.~Math.~Res.~Notices} (1992), no. 10, 199--205.

\bibitem{GMM} J.~Grabowski, G.~Marmo and P.~Michor, Homology and modular classes of Lie algebroids.
  \emph{Ann.~Inst.~Fourier} \textbf{56} (2006), 69--83.

\bibitem{KW} Y.~Kosmann-Schwarzbach and A.~Weinstein, Relative modular classes of Lie algebroids.
  \emph{C.~R.~Math.~Acad.~Sci.~Paris} \textbf{341} (2005), no. 8, 509--514.

\bibitem{KLW} Y.~Kosmann-Schwarzbach, C. Laurent-Gengoux and A.~Weinstein, Modular classes of Lie algebroid morphisms. \emph{Transform.~Groups}  \textbf{13}  (2008),  no. 3-4, 727--755.

\bibitem{Koszul} J.-L.~Koszul, Crochet de Schouten-Nijenhuis et cohomologie.
	\emph{Ast\'erisque Num\'ero Hors S\'erie} (1985), 257--271.
	
\bibitem{Lichnerowicz} A.~Lichnerowicz, Les vari\'{e}t\'{e}s de Poisson et leurs alg\`{e}bres de Lie associ\'{e}es.
  \emph{J.~Diff.~Geom.~}\textbf{12} (1977), 253--300.

\bibitem{Lu} J.-H.~Lu, Multiplicative and Affine Poisson structures on Lie Groups,
    \emph{Ph.D. Thesis}, University of California (Berkeley), 1990.

\bibitem{Lu1} J.-H.~Lu, Momentum mappings and reduction of Poisson actions,
        in \emph{Symplectic geometry, groupoids, and integrable systems} (Berkeley, CA, 1989), 209–-226,
        Math.~Sci.~Res.~Inst.~Publ., 20, Springer, New York (1991).

\bibitem{MM} I.~Moerdijk and J.~Mr\v{c}un, On the integrability of Lie subalgebroids, \emph{Adv.~Math.~}\textbf{204} (1) (2006), 101--115.

\bibitem{Radko} O.~Radko, A classification of topologically stable Poisson structures on a compact oriented surface.
	\emph{J.~Symplectic Geom.~}\textbf{1} (2002), no. 3, 523--542.
	
\bibitem{Vaisman} I.~Vaisman, \emph{Lectures on the geometry of Poisson manifolds}.
	Progress in mathematics, vol.~118, Birkh\"auser, Basel, 1994.

\bibitem{Weinstein1} A.~Weinstein, The modular automorphism group of a Poisson manifold.
  \emph{J.~Geom.~Phys.~}\textbf{23} (1997), no.~3-4, 379--394.

\bibitem{Weinstein2} A.~Weinstein, Coisotropic calculus and Poisson groupoids.
	\emph{J.~Math. Soc.~Japan} \textbf{40} (1988), 705--727.
	
\bibitem{Zung} N.-T.~Zung, Proper groupoids and momentum maps: linearization, affinity, and convexity.
	\emph{Ann.~Sci.~\'Ecole Norm.~Sup.~(4)} \textbf{39} (2006), no.~5, 841--869.

\bibitem{Xu} P. Xu, Dirac submanifolds and Poisson involutions,
	\emph{Ann.~Sci.~\'Ecole Norm.~Sup.~}(4) \textbf{36} (2003), 403--430.

\end{thebibliography}
\end{document}